\documentclass{amsart}

\usepackage[utf8]{inputenc}
\usepackage{microtype}
\usepackage{xspace}
\usepackage{undertilde}
\usepackage{amssymb}
\usepackage[hyperindex,breaklinks,colorlinks=true,linkcolor=black,anchorcolor=black,citecolor=black,filecolor=black,menucolor=black,runcolor=black,urlcolor=blue,pagebackref]{hyperref}

\newtheorem{theorem}{Theorem}[section]
\newtheorem{lemma}[theorem]{Lemma}

\newtheorem{corollary}[theorem]{Corollary}
\newtheorem{claim}[theorem]{Claim}

\theoremstyle{definition}
\newtheorem{definition}[theorem]{Definition}

\theoremstyle{remark}
\newtheorem{remark}[theorem]{Remark}

\numberwithin{equation}{section}

	\usepackage{xcolor}
	\usepackage{soul}
	
	\definecolor{lightblue}{rgb}{.60,.60,1}
	
	\definecolor{ItalianApricot}{rgb}{1,0.7,0.5}

\def\qt#1{``#1''}

\def\eps{\ensuremath{\varepsilon}\xspace}
\def\sC{\ensuremath{\+C}\xspace}

\newcommand{\MLR}{\mathsf{MLR}}

\newcommand{\CR}{\mathsf{CR}}
\newcommand{\DemR}{\mathsf{DemR}}
\newcommand{\DiffR}{\mathsf{DiffR}}
\newcommand{\WR}{\mathsf{W2R}}

\newcommand{\dom}{\text{dom}}
\newcommand{\uh}{{\upharpoonright}}
\newcommand{\uple}[1]{{\langle #1 \rangle}}
\newcommand{\cs}{2^\omega}

\newcommand{\C}{\+C}

\newcommand{\G}{\mathcal{G}}

\newcommand{\N}{\mathbb{N}}

\newcommand{\PP}{\mathbb{P}}
\newcommand{\Q}{\mathbb{Q}}
\newcommand{\R}{\mathbb{R}}

\newcommand{\Cyl}[1]{\ensuremath{[\![{#1}]\!]}}

\newcommand{\DII}{\Delta^0_2}
\newcommand{\NN}{{\mathbb{N}}}
\newcommand{\RR}{{\mathbb{R}}}
\newcommand{\QQ}{{\mathbb{Q}}}

\newcommand{\sub}{\subseteq}
\newcommand{\sN}[1]{_{#1\in \NN}}
\newcommand{\uhr}[1]{\! \upharpoonright_{#1}}
\newcommand{\ML}{Martin-L{\"o}f}
\newcommand{\SI}[1]{\Sigma^0_{#1}}
\newcommand{\PI}[1]{\Pi^0_{#1}}

\newcommand{\bi}{\begin{itemize}}
\newcommand{\ei}{\end{itemize}}
\newcommand{\bc}{\begin{center}}
\newcommand{\ec}{\end{center}}

\newcommand{\Halt}{{\ES'}}
\newcommand{\ES}{\emptyset}
\newcommand{\estring}{\emptyset}
\newcommand{\ria}{\rightarrow}
\newcommand{\tp}[1]{2^{#1}}
\newcommand{\ex}{\exists}
\newcommand{\fa}{\forall}

\newcommand{\Kuc}{Ku{\v c}era}

\newcommand{\strcantor}{2^{ < \omega}}

\newcommand{\n}{\noindent}

\newcommand{\vsps}{\vspace{3pt}}

\newcommand{\vsp}{\vspace{6pt}}
\newcommand{\leb}{\mathbf{\lambda}}
\newcommand{\lwtt}{\le_{\mathrm{wtt}}}

\newcommand{\sss}{\sigma}

\newcommand{\lland}{\, \land \, }

\newcommand\+[1]{\mathcal{#1}}

\newcommand{\ol}{\overline}
\renewcommand{\ul}{\underline}

\newcommand{\lra}{\leftrightarrow}
\newcommand{\LR}{\Leftrightarrow}
\newcommand{\RA}{\Rightarrow}

\newcommand{\sssl}{\ensuremath{|\sigma|}}

\newcommand{\prefx}{\preceq}
\newcommand{\extend}{\succeq}

\begin{document}

\title{Denjoy, Demuth, and Density}


\author{Laurent Bienvenu}
\address{LIAFA, CNRS \& Universit\'e de Paris 7, Case 7014, 75205 Paris Cedex 13, France}
\email{laurent.bienvenu@liafa.jussieu.fr}

\author{Rupert H\"olzl}
\address{Institut für Theoretische Informatik, Mathematik und
Operations Research, Fakultät für Informatik, Universität der Bundeswehr München, Werner-Heisenberg-Weg 39, 85577~Neubiberg, Germany}
\email{r@hoelzl.fr}

\author{Joseph S. Miller}
\address{Department of Mathematics, University of Wisconsin,
Madison, WI 53706-1388, USA}
\email{jmiller@math.wisc.edu}

\author{Andr\'e Nies}
\address{Department of Computer Science, University of Auckland, Private
Bag 92019, Auckland, New Zealand}
\email{andre@cs.auckland.ac.nz}

\thanks{The second author is supported by a Feodor Lynen postdoctoral research fellowship by the Alexander von Humboldt Foundation. The third author is supported by the National Science Foundation under grant DMS-1001847. The fourth author is supported by the Marsden fund of New Zealand.}

\thanks{This article is a significantly extended version of an article published in the Proceedings of the 29th Symposium on Theoretical Aspects of Computer Science (STACS)  2012 by the same authors (volume 14, pages 543--554, 2012).  In the present article we   treat density and   differentiability as topics of equal importance. The previous article was mainly on differentiability, using density as a tool.  The present article also  expands on the STACS article    by adding omitted proofs, simplifying proofs, and expanding proof sketches into full proofs.}

\makeatletter
\@namedef{subjclassname@2010}{\textup{2010} Mathematics Subject Classification}
\makeatother


\date{}

\begin{abstract}
We consider effective versions of two classical theorems, the Le\-besgue density theorem and the Denjoy-Young-Saks theorem. For the first, we show that  a \ML\ random real $z\in [0,1]$ is Turing incomplete if and only if every effectively closed class $\sC \sub [0,1]$ containing~$z$
 has positive density at~$z$. Under the stronger assumption that $z$ is not  LR-hard, we show that $z$  has density-one in every  such class. These results have since been applied to solve two open problems on the interaction between the Turing degrees of \ML\ random reals and $K$-trivial sets: the non-cupping and covering problems.

We say that $f\colon[0,1]\to\R$ satisfies the Denjoy alternative at $z \in [0,1]$ if either the derivative $f'(z)$ exists, or the upper and lower derivatives at $z$ are $+\infty$ and $-\infty$, respectively. The Denjoy-Young-Saks theorem states that every function $f\colon[0,1]\to\R$ satisfies the Denjoy alternative at almost every $z\in[0,1]$. We answer a question posed by \Kuc\ in 2004 by showing that a real $z$ is computably random if and only if every computable function $f$ satisfies the Denjoy alternative at $z$. 

For Markov computable functions, which are only defined on computable reals, we can formulate the Denjoy alternative using pseudo-derivatives. Call a real~$z$ \emph{DA-random} if every Markov computable function satisfies the Denjoy alternative at~$z$. We considerably strengthen a result of Demuth (Comment.\ Math.\ Univ.\ Carolin., 24(3):391--406, 1983) by showing that every Turing incomplete \ML\ random real is DA-random. The proof involves the notion of non-porosity, a variant of density, which is the bridge between the two themes of this paper. We finish by showing that DA-randomness is incomparable with \ML\ randomness.
\end{abstract}

\maketitle

\section{Introduction}

Several important theorems from analysis and ergodic theory assert a certain property for all points outside a null set. Computable analysis and algorithmic randomness both play a role in understanding the effective content of such theorems. This is a very   active area of current  research 
\cite{Brattka.Miller.ea:nd,Freer.Kjos.ea:nd,Pathak.Rojas.ea:12}. The idea is to put an effectiveness hypothesis on the given objects in the theorem and look for a randomness notion that guarantees that a point $z$ satisfies the conclusion of the theorem. Every randomness notion essentially identifies a co-null set of ``random'' points, so it is natural to expect that a sufficiently strong randomness notion avoids exceptions to the (countably many) effective cases of the analytical theorem. Frequently, we can even characterize randomness classes using effective forms of well-known theorems from analysis. We illustrate with three examples.

\begin{enumerate}
\item   A well-known theorem of  Lebesgue  states that every  nondecreasing  function $f\colon[0,1]\to\R$  is almost everywhere differentiable (see, e.g., \cite[Thm.\ 20.1]{Carothers:00}).  

\item Any function of bounded variation is the difference of two non-decreasing functions (Jordan), so  it is   almost everywhere differentiable by (1).

%

\item The  Lebesgue differentiation theorem states that if $g\in L^1([0,1]^d)$, then for almost every $z\in[0,1]^d$, \bc $g(z) = \lim_{Q  \to z}  \frac{1}{\lambda(Q)}\int_{Q} g$, \ec  where $Q$ is an open cube containing $z$ with volume  $\lambda(Q)$ tending to $0$.

\end{enumerate}

\noindent Each of these theorems has been effectivized to characterize a standard randomness class. The most successful definition of randomness for infinite binary sequences was given by Martin-L\"of~\cite{MartinLof1966}, but several other definitions have appeared (see Section~\ref{ss:rdness notions}). Schnorr \cite{Sc:71a,Sc:71} introduced two weaker notions, now called \emph{computable randomness} and \emph{Schnorr randomness}. \ML\ randomness properly implies computable randomness, which in turn properly implies Schnorr randomness. Although these definitions were given for infinite binary sequences, they can be applied to real numbers via the binary expansion.  The fact that they  each have analytic characterizations gives new evidence that they are natural notions.

\begin{enumerate}
\renewcommand{\labelenumi}{(\alph{enumi})}

\item Brattka, Miller and Nies~\cite{Brattka.Miller.ea:nd} showed that $z$ is computably random iff every computable non-decreasing function is differentiable at~$z$. This result was surprising because previously, computable randomness had no characterizations other than its original definition, and in particular, no analytical characterization.

\item Demuth \cite{De:75} showed that $z$ is \ML\ random iff every computable function with bounded variation is differentiable at~$z$. See \cite{Brattka.Miller.ea:nd} for a proof.

\item Pathak, Rojas, and Simpson \cite{Pathak.Rojas.ea:12} and Rute \cite{Rute:12} proved that $z\in[0,1]^d$ is Schnorr random iff for every $L_1$-computable function $g \colon \, [0,1]^d \to \RR$,  the     $\lim_{Q  \to z}  \frac{1}{\lambda(Q)}\int_{Q} g$   exists. This is an effective version of the   Lebesgue differentiation theorem, but takes into account only the existence of limits.

\end{enumerate}

Other results in this line of research include the characterization of  both Schnorr randomness \cite{GaHoRo:11} and \ML\ randomness \cite{BienvenuDHMS2011, FranklinGMN2011} using effective versions of Birkhoff's ergodic theorem.

We study two theorems from analysis of the ``almost everywhere'' kind, the \emph{Lebesgue density theorem} (see Sections~\ref{subsec:intro-density} and~\ref{sec:density}) and the striking \emph{Denjoy-Young-Saks theorem}, which says that \emph{every} function $f\colon[0,1]\to\R$ satisfies the Denjoy alternative at almost every $z\in[0,1]$ (see Sections~\ref{subsec:intro-denjoy} and~\ref{sec:denjoy}). We characterize the computably random reals as those for which the Denjoy-Young-Saks theorem holds for computable functions (Theorem~\ref{thm:DenjoyCR}). Otherwise, we consider effectivizations for which even \ML\ randomness is not sufficient to make the theorem hold at a real, setting our results apart from earlier work.

\emph{Difference randomness} of a real, which is slightly stronger than \ML\ randomness, will play an important role. It was introduced by Franklin and Ng \cite{FranklinNg}, who showed that a real is difference random if and only if it is \ML\ random and Turing incomplete. We characterize difference randomness  of a real in terms of its density in   effectively closed classes containing the real (Theorem~\ref{thm:density_Turing}) and show that difference randomness is sufficient to make the Denjoy-Young-Saks theorem hold for all Markov computable functions (Section~\ref{subsec:diff-Denjoy}). The relationship between density and difference randomness has unexpected applications  to the study of $K$-triviality and its interaction with randomness; we give detail in Section~\ref{subsec:discussion}.

\subsection{Lebesgue Density}
\label{subsec:intro-density}

We will  first discuss the Lebesgue density theorem \cite[page 407]{Lebesgue:1910}. Let~$\lambda$ denote Lebesgue measure.
\begin{definition}
We define the (lower Lebesgue) density of a set $\sC \subseteq \R$ at a point~$z$ to be the quantity $$\varrho(\sC | z):=\liminf_{\gamma,\delta \rightarrow 0^+} \frac{\lambda([z-\gamma,z+\delta] \cap \sC)}{\gamma + \delta}.$$
\end{definition}

Intuitively, this measures the  fraction of space  filled by~$\sC$ around~$z$ if we ``zoom in'' arbitrarily close. Note that  $0 \le \varrho(\sC | z) \le 1$.

\begin{theorem}[Lebesgue density theorem]
Let $\sC \subseteq \R$ be a measurable set. Then  $\varrho(\sC | z)=1$ for almost every $z\in \sC$.
\end{theorem}
The Lebesgue differentiation theorem implies the Lebesgue density theorem: let $g$ be the characteristic function~$1_\sC$ and note that the limit of averages around $z$ equals the density of $z$ in $\sC$. 

If $\sC$ is open, the theorem is trivial, and in fact holds for all $z\in \sC$. As pointed out by T.\ Tao in his blog \cite{MR2459552}, the Lebesgue density theorem implies that general measurable sets behave similarly to open sets. The simplest non-trivial case of the theorem is for closed sets. We will consider the effective case when $\sC$ is an \emph{effectively closed} class, in other words, the complement of $\sC$ is is the  union of  an effective   sequence  of open intervals with rational endpoints.

\begin{definition}
Consider $z\in[0,1]$.
\begin{itemize}
\item We say that $z$ is a \emph{density-one point}   if $\varrho(\sC | z)=1$ for every effectively closed class $\sC$ containing $z$.  

\item We say that $z$ is a \emph{positive density point} if $\varrho(\sC | z)>0$ for every effectively closed class $\sC$ containing $z$.
\end{itemize}
\end{definition}
By the Lebesgue density theorem and the fact that there are only countably many effectively closed classes, almost every $z$ is a density-one point. 

A third notion closely related to density, \emph{non-porosity}, will be crucial in the proofs of Theorems~\ref{thm:2} and~\ref{thm:1}. The definition originates in the work of Denjoy. See for instance \cite[Ex.\ 7:9.12]{bruckner2thomson}, or \cite[5.8.124]{Bogachev.vol1:07} (but note the typo in the definition there). 

\begin{definition} \label{def:porous at} We say that a set $\sC\sub\mathbb R$ is \emph{porous at} a real $z$ via $\eps>0$ if for each $\alpha>0$, there exists $\beta$ with $ 0<\beta\le \alpha$ such that $(z-\beta, z+ \beta)$ contains an open interval of length $\varepsilon\beta $ that is disjoint from $\sC$. We say that $\sC$ is \emph{porous at} $z$ if it is porous  at $z$ via some~$\varepsilon>0$. \end{definition}

\begin{definition}
We call $z$ a \emph{non-porosity point} if every effectively closed class to which it belongs is non-porous at $z$.
\end{definition}

Clearly, porosity of $\+C$ at $z$ implies that $\varrho(\sC | z)<1$. Therefore, almost every $z$ is a non-porosity point.

In contrast to examples (a)--(c) above, we cannot characterize density-one points in terms of an algorithmic randomness notion. The reason is that every $1$-generic real $z$ is a density-one point: an effectively closed class $\sC$ contains a $1$-generic $z$ only if $\sC$ contains an open interval around $z$. But no $1$-generic is Martin-L\"of random, or even Schnorr random. Indeed, $1$-generics violate basic properties we expect of random sequences, such as the law of large numbers. (For more information, see either \cite{DowneyH2010} or \cite{Nies2009}.) To give us some hope of understanding density from the standpoint of algorithmic randomness, we restrict our attention to reals $z$ that are Martin-L\"of random.

We will not fully characterize the Martin-L\"of random reals that are density-one points. Instead, we give several partial results, including a characterization of the Martin-L\"of random positive density points. The following diagram summarizes our results pertaining to density. Assuming that $z$ is Martin-L\"of random:

\begin{center}
\newcommand{\specialcell}[2][c]{%
  \begin{tabular}[#1]{@{}c@{}}#2\end{tabular}}

\medskip
\begin{tabular}{rcccccl}
\specialcell{$z$ is not\\LR-hard} &
	\hspace*{-0.1cm}$\xrightarrow{\text{Thm.~\ref{thm:aed}}}$\hspace*{-0.1cm} &
\specialcell{$z$ is a density-\\one point} &
	\hspace*{-0.2cm}$\xrightarrow{\hspace*{0.5cm}}$\hspace*{-0.2cm} &
\specialcell{$z$ is a positive\\density point} \\
	& & & & \rule{0cm}{16px}{\scriptsize Thm.~\ref{thm:density_Turing}}$\Big\updownarrow$\hspace*{1cm} \\
	& & & & \specialcell{$z\ngeq_T\emptyset'$} &
	\hspace*{-0.6cm}$\xrightarrow{\text{Lem.~\ref{lem:porous}}}$\hspace*{-0.1cm} &
\specialcell{$z$ is a non-\\porosity point.}
\end{tabular}
\end{center}
\medskip

{LR-hardness} is a slight weakening of Turing completeness discussed in Section~\ref{subsec:density}. By \cite[6.3.13]{Nies2009}, there is a $\DII$ ML-random real that is LR-hard but not Turing complete. Very recently, Day and Miller \cite{Day.Miller:nd} have shown that there is a Martin-L\"of random positive density point that is not a density-one point. The converses of the leftmost and rightmost implications are open.

\subsection{Discussion}
\label{subsec:discussion}

We have given examples of the application of algorithmic randomness to effective forms of ``almost everywhere'' theorems in analysis. This line of research   should also  help to   understand randomness classes better, and should lead to new results in algorithmic randomness. Until recently, the best example of this was   the base invariance of computable randomness \cite{Brattka.Miller.ea:nd}, which follows from the analytic characterization of computable randomness. But the work on density has led to some exciting developments and has helped answer long-standing questions about the connections  between randomness and computability.
We now discuss three very recent applications of this work on density. The first  two  involve \emph{$K$-triviality}. There are many characterizations of $K$-triviality; the one we will use is that a $K$-trivial set $A$ is \emph{low for random}, meaning that every Martin-L\"of random real is Martin-L\"of random relative to $A$.\footnote{See either \cite{DowneyH2010} or \cite{Nies2009} for a thorough introduction to $K$-triviality.} 
  
\medskip
\noindent (1)
In Theorem~\ref{thm:density_Turing}, we will show that if $z$ is \ML\ random, then $z$ is a positive density point iff $z\ngeq_T\emptyset'$. This result was recently used by Day and Miller~\cite{DM} to solve the \emph{non-cupping} problem: an open question of G\'acs about the $K$-trivial sets (see \cite[Question 4.8]{MN06}). They showed that a set $A$ is $K$-trivial $\LR$ $A \oplus z \ge_T \Halt$ implies $z\ge_T \Halt$ for every ML-random real $z$. The characterization of incomplete ML-random reals was essential in the proof of the ``$\RA$'' direction.

\medskip
\noindent (2)
The \emph{covering} question originated in  \cite{Hirschfeldt.Nies.ea:07}: is every $K$-trivial set $A$   Turing below an incomplete ML-random real $z$?  (Also see \cite[Question 4.6]{MN06}.) Using density and related notions,  Bienvenu, Greenberg, \Kuc, Nies and Turetsky~\cite{Bienvenu.Greenberg.ea:OWpreprint}   built a $K$-trivial set $A$ such that any such $z$ is necessarily LR-hard. They also showed that if a Martin-L\"of random real $z$ is not a density-one point, then $z$ computes every $K$-trivial. (See Theorem~\ref{thm:ML-rd_nondensity} for an alternative proof.) As mentioned above, Day and Miller \cite{Day.Miller:nd} constructed a Turing incomplete ML-random that is not a density-one point. Together, these results establish the existence of a single incomplete ML-random computing all $K$-trivials. In particular, this settles the  covering question \cite[Question 4.6]{MN06} in the affirmative.


\medskip
\noindent (3)
The third application involves DNC functions. We say that $g\colon\N\to\N$ is \emph{diagonally non-computable} (\emph{DNC}) if $(\forall n)\; g(n)\neq\varphi_n(n)$, where $(\varphi_e)_{e\in\N}$ is an effective numbering of the partial computable functions. Ku{\v c}era proved that every \ML\ random computes a DNC function. The question then arises  how slowly   such a function can grow. It is not hard to show that if $f\colon\N\to\N$ is computable and $\sum_{n\in\N} \frac{1}{f(n)+1}<\infty$, then every \ML\ random computes an $f$-dominated DNC function. Under a reasonable assumption on the numbering of partial computable functions, Miller \cite{Mi:} proved that if $f\colon\N\to\N$ is a non-decreasing computable function for which $\sum_{n\in\N} \frac{1}{f(n)+1}=\infty$, then a \ML\ random real $z$ computes an $f$-dominated DNC function iff $z\geq_T\emptyset'$. The proof uses Theorem~\ref{thm:density_Turing}, and in fact, is a modification of our original proof of that theorem (see \cite[Theorem 20]{BHNM_diff}).

\subsection{The Denjoy alternative}
\label{subsec:intro-denjoy}

We say that a function  $f\colon [0,1] \to \RR$ satisfies the \emph{Denjoy alternative} at a real $z$ if either $f'(z)$  exists, or $\ol D f(z)= \infty$ and $\ul Df(z) = -\infty$. Here $\ol D f$ and $\ul Df$ are the upper and lower derivatives, respectively (see page \pageref{def_upper_lower_deriv}). The following result is due to Denjoy in its original form for continuous functions, and was successively improved by Young and Saks to the general case.

\begin{theorem}[Denjoy-Young-Saks theorem\footnotemark]
Let $f\colon [0,1] \to \RR$ be an arbitrary function. Then $f$ satisfies the Denjoy alternative at almost every $z\in [0,1]$.
\end{theorem}

\footnotetext{The full version (see, e.g., \cite{Bogachev.vol1:07}) of the Denjoy-Young-Saks theorem is stated using the four Dini derivatives, the upper and lower derivatives from the right and the left, and makes more case distinctions than we do here. The compact form that we have stated is possible because we are using the two-sided upper and lower derivatives.}

Note that the Denjoy-Young-Saks theorem trivially implies Lebesgue's theorem on the a.e.\ differentiability of non-decreasing functions because $\ul Df(z) \ge 0$ for any non-decreasing function $f$.

We answer a question posed by \Kuc\ at the 2004 Logic, Computability and Randomness meeting in C\'ordoba, Argentina by showing that a real $z$ is computably random iff every computable function $f\colon [0,1] \to \RR$ satisfies the Denjoy alternative at $z$. Just as in the classical Denjoy-Young-Saks theorem, in the effective setting we do not need any analytical hypothesis on the function.  This contrasts with the characterization of computable randomness in Brattka, Miller and Nies~\cite{Brattka.Miller.ea:nd}, where the functions have to be non-decreasing. (Note, however, that every computable function on the reals is continuous.)

The first  connections between algorithmic randomness notions and analysis were made by the Czech constructivist Osvald Demuth in the 1970s and 80s. See the survey \cite{Kucera.Nies:12} for background on Demuth's work. In \cite{Demuth:82a}, Demuth introduced a randomness notion that we now call \emph{Demuth randomness}.  In~\cite{Demuth:88}, he studied the Denjoy alternative for Markov computable functions. A real-valued function $f$ defined on the computable reals is called \emph{Markov computable} if, roughly, from a computable name for a real $z$, one can effectively determine a computable name for $f(z)$. (See Section~\ref{s:prelims_analysis} for more details.) We say that a real~$z$ is \emph{DA-random} if every Markov computable function satisfies the Denjoy alternative at~$z$. Since a Markov computable function $f$ may be partial, we need to express the Denjoy alternative using pseudo-derivatives that take into account only the points near $z$ where $f$ is defined (see page~\pageref{def_pseudo_deriv}). If $f$ is total and continuous, the pseudo-derivatives coincide with the usual derivatives.

Note that DA-randomness implies computable randomness by the result of Brattka, Miller and Nies~\cite{Brattka.Miller.ea:nd} mentioned in (b) above. Thus, DA-randomness can be considered a true randomness notion (in contrast to being a density-one point). Demuth~\cite{Demuth:88} proved that Demuth randomness implies DA-randomness. We show that difference randomness, which is significantly weaker than Demuth randomness, implies DA-randomness. Thus Demuth randomness is too strong a hypothesis. 
\begin{theorem}\label{thm:2}
Every difference random real is DA-random.
\end{theorem}
A function $f$ defined on all computable reals is Banach-Mazur computable if it maps computable sequences of reals to computable sequences of reals (but where the mapping does not have to be effective). In fact, our proof yields the Denjoy alternative at difference random reals for the class of Banach-Mazur computable functions~$f$. Hertling \cite{MR2110825} proved that this class of functions is strictly larger than the class of Markov computable functions.

The even weaker effectiveness requirement, that $f(q)$ is computable uniformly in a rational $q \in [0,1]$, seems to be insufficient: our proof uses continuity at each computable real, which may fail for such functions. This contrasts with the case of Brattka, Miller and Nies~\cite[Thm.\ 7.3]{Brattka.Miller.ea:nd}, where such a requirement on a non-decreasing function suffices to establish differentiability at each computably random real. 
However, we show in Theorem~\ref{thm:DA_w2r}, which leads up to Theorem~\ref{thm:2}, that the stronger notion of weak $2$-randomness is sufficient for such functions. 

It turns out that Martin-L\"of randomness is neither sufficient nor necessary to ensure the Denjoy alternative for Markov computable functions.
\begin{theorem}\label{thm:1}
DA-randomness is  incomparable with Martin-L\"of randomness.
\end{theorem}
This is the first time that a reasonably natural randomness notion is incomparable with \ML's. These results will be proven in Section~\ref{sec:denjoy}.

\section{Preliminaries: notation and randomness notions}

\subsection{Basic notation}
\label{ss:basic}
The set of finite binary sequences (we also say strings) is denoted by $\strcantor$, and the set of infinite binary sequences, called Cantor space, is denoted by $\cs$. We write $|\sigma|$ for the length of a string $\sigma$. If $\sigma$ is a string and~$z$ is either a string or an infinite binary sequence, we say that $\sigma$ is a prefix of~$z$, which we write $\sigma \prefx z$, if the first $|\sigma|$ bits of~$z$ are exactly the string~$\sigma$. Given a binary sequence, finite or infinite, with length at least~$n$, $z \uh n$ denotes the string made of the first~$n$ bits of~$z$.

Cantor space is classically endowed with the product topology over the discrete topology on $\{0,1\}$. A basis of this topology is the set of cylinders: given a string $\sigma \in \strcantor$, the cylinder $\Cyl{\sigma}$ is the set of elements of $\cs$ having $\sigma$ as a prefix. If $A$ is a set of strings, $\Cyl{A}$ is the union of the cylinders $\Cyl{\sigma}$ with $\sigma \in A$. 
The Lebesgue measure~$\lambda$ (or uniform measure) on the Cantor space is the probability measure assigning to each bit the value $0$ with probability $1/2$ and the value $1$ with probability $1/2$, independently of all other bits. Equivalently it is the measure~$\lambda$ such that $\lambda(\Cyl \sigma)=2^{-|\sigma|}$ for all~$\sigma$. We abbreviate $\lambda(\Cyl \sigma)$ by $\lambda(\sigma)$. Given two subsets $X$ and $Y$, the second one being of positive measure, the conditional measure $\lambda_Y(X)$ of $X$ knowing~$Y$ is the quantity $\lambda(X\cap Y)/\lambda(Y)$. As before, if $X$ or $Y$ is a cylinder $\Cyl{\sigma}$, we will simply write it as $\sigma$.

Most of the paper will focus on functions from $[0,1]$ to $\R$. The set $[0,1]$ is typically identified with $\cs$, where a real~$z \in [0,1]$ is identified with its binary expansion. This expansion is unique, except for dyadic rationals (i.e., rationals of the form $a 2^{-b}$ with $a,b$ positive integers) which have two. A cylinder $\Cyl \sigma$ will commonly be identified with the open interval $(0.\sigma,0.\sigma+2^{-n})$, where $0.\sigma$ is the dyadic rational whose binary extension is $0.\sigma000\dots$.

An open set in $\cs$ or $[0,1]$ is a union of cylinders. If it is a union of a \emph{computably enumerable} (c.e.) family of cylinders, it is said to be \emph{effectively open} (or \emph{c.e.\ open}). A set is called \emph{effectively closed} if its complement is effectively open.

\subsection{Randomness notions}
\label{ss:rdness notions}

We review the algorithmic randomness notions needed below. For background on most of them see \cite[Chapter 3]{Nies2009} or \cite{DowneyH2010}. We give the definitions in Cantor space, which is standard, but we primarily apply these definitions to real numbers. There are two equivalent ways to translate between these spaces. We can either transfer the definitions, which is easily done in most cases,\footnote{One exception to this rule is computable randomness, which is defined via computable betting strategies processing the bit sequence.} or we can say that $z\in\R$ is in a certain randomness class if (the tail of) its binary expansion is in that class. Dyadic rationals are never random, of course.

Recall that a sequence of open classes $(U_n)\sN n $ is said to be \emph{uniformly c.e.}\ if there is a sequence $(B_n)\sN n$ of uniformly c.e.\ sets of strings such that each~$U_n=\Cyl{B_n}$. In Cantor space we may assume that each $B_n$ is an antichain under the prefix relation of strings. A \emph{Martin-L\"of test} is a uniformly c.e.\ sequence $(U_n)\sN n$ of open classes such that for all~$n$, $\lambda(U_n) \le  2^{-n}$. A sequence $z \in \cs$ is called \emph{Martin-L\"of random} if   for any Martin-L\"of test $(U_n)_n$ we have $z \not\in \bigcap_n U_n$. Note that we can replace $2^{-n}$ in the definition of Martin-L\"of test with any positive computable function that limits to zero without changing the randomness notion.

A \emph{Solovay test} is a uniformly c.e.\ sequence $(S_n)\sN n$ of open classes such that $\sum_n \leb(S_n)<\infty$. A Solovay tests \emph{captures} $z\in\cs$ if $z$ is in infinitely many of the $S_n$. It can be shown that $z$ is Martin-L\"of random if and only if there is no Solovay test capturing it (see for instance \cite{DowneyH2010}). Note that every Martin-L\"of test $(U_n)\sN n$ is also a Solovay test, so a \ML\ random $z$ can belong to at most finitely many $U_n$.

A generalized ML-test is a uniformly c.e.\ sequence $(U_n)\sN n$ of open classes with the weaker condition that $\lim_{n} \lambda(U_n) = 0$. A sequence $z \in \cs$ is \emph{weakly $2$-random} if there is no generalized ML-test capturing it. This is the same as saying that $z$ is not contained in any null $\Pi^0_2$ class. It is clear that weak $2$-randomness is at least as strong as Martin-L\"of randomness; it is not hard to prove that it is strictly stronger.

An important randomness notion for this paper is difference randomness, which lies strictly between Martin-L\"of and weak $2$-randomness. The following is not identical,  but equivalent to the original definition by Franklin and Ng~\cite{FranklinNg}.
\begin{definition}\label{def:diff_test}
A \emph{difference test} $((U_n)\sN n,\C)$ consists of a uniformly c.e.\ sequence $(U_n)\sN n$ of open classes and a single effectively closed class $\C$ such that for all~$n$, $\lambda(U_n\cap\C) \le  2^{-n}$. A sequence $z \in \cs$ is called \emph{difference random} if there is no difference test \emph{capturing} it, i.e., if for any difference test $((U_n)\sN n,\C)$ we have $z\not\in \bigcap_{n\in \NN}  (U_n \cap \C)$.
\end{definition}

Franklin and Ng~\cite{FranklinNg} proved that a  real is difference random if and only if it is  Martin-L\"of random and Turing incomplete (i.e., does not compute $\emptyset'$).

Another strengthening of Martin-L\"of randomness is Demuth randomness. A~function $f\colon\N \rightarrow \N$ is called $\omega$-c.e.\ if $f \lwtt \Halt$. A \emph{Demuth test} is a sequence $(U_n)$ of effectively open classes with $\lambda(U_n) \le 2^{-n}$ for all $n$, but it is not necessarily uniformly c.e. Instead it satisfies the following weak form of uniformity: there exists an $\omega$-c.e.\ function $f\colon\N \rightarrow \N$ that for each~$n$ gives a c.e.\ index for a set of strings generating~$U_n$.

\begin{definition}
A sequence~$z \in \cs$ is said to be \emph{Demuth random} if for every Demuth test $(U_n)\sN n$, $z$ belongs to only finitely many $U_n$.
\end{definition}

The last notion of randomness we will discuss in the paper is computable randomness. Its definition involves the notion of martingale.

\begin{definition}
A martingale is a function~$d\colon \strcantor \rightarrow [0,\infty)$ such that for all~$\sigma \in \strcantor$
$$
d(\sigma)=\frac{d(\sigma0)+d(\sigma1)}{2}.
$$
\end{definition}

Intuitively, a martingale represents a betting strategy where a player successively bets money on the values of the bits of an infinite binary sequence (doubling its stake when the guess is correct); $d(\sigma)$ then represents the capital of the player after betting on initial segment~$\sigma$. With this intuition, a martingale succeeds against a sequence~$z$ if $\limsup_n d(z \uh n) = +\infty$. A computably random sequence is a sequence against which no computable betting strategy succeeds. In other words:

\begin{definition}
A sequence~$z \in \cs$ is \emph{computably random} if and only if for every computable martingale~$d$, $\limsup_n d(z \uh n) < +\infty$.
\end{definition}

We denote by $\MLR$, $\WR$, $\DiffR$, $\DemR$, $\CR$ the classes of Martin-L\"of random, weakly $2$-random, difference random, Demuth random and computably random sequences respectively.\\

Given a sequence~$z \in \cs$, the following implications
$$
\begin{array}{rcccccc}
z \in \WR & & & & & \\
 & \searrow & & & &\\
 & & z \in \DiffR  & \longrightarrow & z \in \MLR & \longrightarrow & z \in \CR \\
 & \nearrow & & & & \\
  z \in \DemR\\
\end{array}$$
hold, and no other implication holds in general (other than those which can be derived by transitivity from the above diagram). See, for example,~\cite{Nies2009} for a detailed exposition.

\section{Density}
\label{sec:density}

\subsection*{Covering non density-one  points}

We want to  show that sufficiently random reals are positive density points, or even density-one points. To do so, we must cover points of low density with sets of small measure. We give two bounds.

\begin{lemma}\label{lem:bounds}
Let $\+C\subseteq[0,1]$ be a closed set. Fix $\eps\in(0,1)$ and let
\[
U = \{z\mid\exists \text{ open interval } I\colon z\in I \text{ and } \leb_I(\+C)<\eps\}.
\]
Then \textnormal{(1)} $\leb(\+C\cap U)\leq 2\eps$,\quad and \textnormal{(2)} $\leb(U)\leq 2\frac{1-\leb(\+C)}{1-\eps}$.
\end{lemma}
\begin{proof}
Both bounds would be easy to prove (even without the factors of $2$) if we could cover $U$ with an anti-chain of intervals $J$ such that $\leb_J(\+C)\leq\eps$. This would be possible if we were working in Cantor space (with the corresponding density notion), but it is not always possible on the real interval. Most of the work we do below is to get around this obstacle by producing two such anti-chains that together cover $U$, accounting for the extra factors of $2$.

First note that, without loss of generality, we may assume that the complement of $\+C$ is the disjoint union of a finite collection $\+ S$ of open intervals. To see this, let $\+C = \bigcap_{s\in\N} \+C_s$, where $\+C_0\supseteq \+C_1\supseteq\cdots$ is a nested sequence of closed sets whose complements each consist of a finite number of intervals. Then $U = \bigcup_{s\in\N} U[\+C_s]$, where $U[\+C_s]$ is defined analogously to $U$ with $\+C_s$ in place of $\+C$. Note that $U[\+C_s]$ is an increasing nested sequence of open sets and that $\+C\cap U[\+C_s]\subseteq \+C_s\cap U[\+C_s]$. Hence, bounding each $\leb(\+C_s\cap U[\+C_s])$ by $2\eps$ would prove (1). Because $\leb(\+C_s)$ converges to $\leb(\+C)$, bounding each $\leb(U[\+C_s])$ by $2(1-\leb(\+C_s))/(1-\eps)$ would prove (2).

Call a non-trivial interval $I\subseteq[0,1]$ \emph{fat} if:
\begin{itemize}
\item[(i)] $\leb_I(\+C)\leq\eps$.
\item[(ii)] Either the left endpoint of~$I$ is also the left endpoint of some interval in $\+S$, or the right endpoint of~$I$ is also the right endpoint of some interval in $\+S$.
\item[(iii)] $I$ is maximal (w.r.t.\ inclusion) within the family of intervals having properties (i) and (ii).
\end{itemize}

\noindent {\em Claim $1$.} Any interval $I$ satisfying (i) and (ii) is contained in some fat interval.\\ 
\noindent {\em Subproof.}
Let $=[a_0,b_0],\dots,[a_N,b_N]$ be the elements of $\+S$. For each $i\leq N$, let 
\begin{align*}
	a'_i &= \min \{ x \leq b_i \, \mid \, \lambda_{[x,b_i]}(\+ C) \leq \eps\}\text{, and}\\
	b'_i &= \max \{ x \geq a_i \, \mid \, \lambda_{[a_i,x]}(\+ C) \leq \eps\}.
\end{align*}
Define $\+ F_0$ to be the family of intervals $\{ [a'_i,b_i] \, \mid \, i \geq 0\}  \cup \{ [a_i,b'_i] \, \mid \, i \geq 0\}$ and let $\+ F$ be the maximal elements of $\+ F_0$. Observe that $\+ F$ is in fact the family of fat intervals. Indeed, if $I=[z,b_i]$ is fat, then $z=a'_i$ by maximality, so $I \in \+ F_0$. Applying maximality again, $I \in \+ F$. The same argument holds if $I=[a_i,z]$ is fat. Conversely, if $I \in \+ F$ were not fat, it would be strictly contained in an interval $I'$ of type $[x,b_i]$ or $[a_i,x]$ such that $\lambda_{I'}(\+ C) \leq \eps$, and therefore be strictly contained either in $[a'_i,b_i]$ or $[a_i,b'_i]$ which are both in $\+ F_0$. This contradicts the maximality of $I$ inside $\+ F_0$. The same argument shows that if an interval~$I$ satisfies properties (i) and (ii) and is not fat, it is contained in a maximal element of $\+ F$, i.e., in a fat interval. 
\hfill $\Diamond$

\vsp

\noindent {\em Claim $2$.} Every point $x \in U$ is contained in some fat interval. \\
\noindent {\em Subproof.}
If $x$ is in the complement of $\+ C$, then it is contained in some $(a_i,b_i) \in \+ S$ and the result is clear. So assume that $x \in U \cap \+ C$. By definition, there exists an interval $I$ containing~$x$ such that $\leb_I(\+C)\leq\eps$, which in particular implies that~$I$ contains some $J \in \+ S$, to which~$x$ does not belong. Without loss of generality suppose that~$J$ is on the right of~$x$ and that $J$ is the rightmost interval of $\+ S$ contained in $I$. Consider the interval~$I'$ obtained by shifting (keeping length constant) the interval~$I$ to the left so that the right endpoint of~$I'$ is the right endpoint of~$J$. Shifting $I$ to $I'$ amounts to removing some interval on the left and adding an interval of the same length on the right. The removed interval is entirely contained in $\+ C$ (this is because~$J$ is the rightmost interval of $\+ S$ contained in~$I$), therefore $\lambda_{I'}(\+ C) \leq \lambda_{I}(\+ C)\leq \eps$. So~$I'$ satisfies conditions (i) and (ii) and still contains~$x$. By our first claim, there is a fat interval $I''$ containing~$I'$, proving our second claim. \hfill $\Diamond$

\vsp

For  any two intervals $I, J$, we write $I \preceq J$ if $\min(I) \le \min (J)$, and $I \prec J$ if $I \preceq J$ and $I \neq J$. If $I \prec J$ we say that $I$ is \emph{to the left of} $J$ and that $J$ is \emph{to the right of} $I$. By definition, no fat interval $I$ can properly contain another fat interval~$J$. So for fat intervals, $I \preceq J$ implies $\max (I) \le \max (J)$. Therefore, on fat intervals, the linear order $\preceq$ defined by the order of their left endpoints is the same as one defined by the order of their right endpoints.

Now, build a sequence  $I_0 \prec  I_1 \prec  \dots$ of intervals as follows. Let $I_0$ be the leftmost fat interval. Suppose that  $I_n$ has been defined.  \textit{(i)} If there is a fat interval $J \succ I_n$ that intersects $I_n$, then let $I_{n+1}$ be the rightmost such interval.
\textit{(ii)} If not, but there is a fat interval $ J \succ I_n$, let $I_{n+1}$ be the leftmost such interval. 
\textit{(iii)} Otherwise, terminate the sequence. We make the following observations. 

\bi
\item[(a)] \emph{The sequence terminates} because there are only finitely many fat intervals. Let $I_N$ be the last interval in the sequence.

\item[(b)]\emph{It is not possible for $I_n$ to intersect $I_{n+2}$,} because otherwise $I_{n+1}$ would not have been the rightmost fat interval intersecting $I_n$.


\item[(c)] \emph{If $I$ is a fat interval, then $I\subseteq \bigcup_{n\leq N} I_n$.} 
By choice of $I_0$, there is a largest $r$ such that $\min(I_r) \leq \min(I)$. If $I=I_r$ we are done, so assume otherwise. If $I_r \cap I_{r+1}\neq \emptyset$, then $I \subseteq I_r \cup I_{r+1}$. On the other hand, If $I_r \cap I_{r+1} = \emptyset$, then $I_r \cap I = \emptyset$, because no fat interval to the right of $I_r$ intersects $I_r$. But then $\min(I_{r+1}) \leq \min(I)$ by the choice of $I_{r+1}$ as the leftmost fat interval to the right of $I_r$, contradicting the choice of $r$. \ei

Note that by (b), the even indexed members of $I_0, I_1, \dots, I_N$ are a disjoint sequence of fat intervals, as are the odd indexed members. By (c) and the fact that every $z\in U$ is contained in a fat interval, $U\subseteq \bigcup_{n\leq N} I_n$. We are ready to prove (1) and (2). 

For (1), we will show that 
$\leb(\+C\cap\bigcup_{2m\leq N} I_{2m})\leq \eps$ and $\leb(\+C\cap\bigcup_{2m+1\leq N} I_{2m+1})\leq \eps$. The first equation for the even subsequence can be shown as follows:
 \[
\leb\Bigg(\+C\cap \bigcup_{2m\leq N} I_{2m}\Bigg) = \sum_{2m\leq N} \leb(\+C\cap I_{2m}) = \sum_{2m\leq N} |I_{2m}|\;\leb_{I_{2m}}(\+C) \leq \sum_{2m\leq N} |I_{2m}|\;\eps \leq \eps.
\]
The proof for the second equation is analogous, completing the proof of (1).

For (2), we first show that $\leb(\bigcup_{2m\leq N} I_{2m})\leq (1-\leb(\+C))/(1-\eps)$:
\begin{multline*}
\leb\Bigg(\bigcup_{2m\leq N} I_{2m}\Bigg)(1-\eps) = \sum_{2m\leq N} |I_{2m}|\;(1-\eps) \leq \sum_{2m\leq N} |I_{2m}|\;\leb_{I_{2m}}([0,1]\smallsetminus\+C) \\
= \sum_{2m\leq N} \leb(([0,1]\smallsetminus\+C)\cap I_{2m}) = \leb\Bigg(([0,1]\smallsetminus\+C)\cap\bigcup_{2m\leq N} I_{2m}\Bigg) \leq \leb([0,1]\smallsetminus\+C) = 1-\leb(\+C).
\end{multline*}
Analogously, we see that $\leb(\bigcup_{2m+1\leq N} I_{2m+1})\leq (1-\leb(\+C))/(1-\eps)$, proving (2).
\end{proof}

\subsection{Positive density points}

Recall that Franklin and Ng~\cite{FranklinNg} showed that $z$ is \emph{difference random} iff it is Martin-L\"of random and Turing incomplete. In this section we use the notion of lower density for effectively closed classes to give an analytic characterization of difference randomness.
\begin{theorem}\label{thm:density_Turing}
Let $z$ be a Martin-L\"of random real. Then $z\ngeq_T\emptyset'$ iff $z$ is a positive density point.
\end{theorem}
As mentioned in the introduction, this result was recently used by Day and Miller~\cite{DM} to answer an open question about the $K$-trivial sets (see \cite[Question 4.8]{MN06}). The theorem follows immediately from Lemma~\ref{lem:density}. If $z$ is not a positive density point, this is witnessed by an effectively closed class. Similarly, if $z$ is not difference random, the test covering $z$ includes an effectively closed class. The lemma says that as long as $z$ is Martin-L\"of random, the classes witnessing both properties are the same.

\begin{lemma}\label{lem:density}
Let $z$ be a Martin-L\"of random real. Let $\+C$ be an effectively closed class containing $z$. The following are equivalent:
\begin{enumerate}
\item[(i)] $z$ fails a difference test of the form $((U_n)\sN n,\+C)$.
\item[(ii)] $z$ has lower Lebesgue density zero in $\+C$, i.e., $\varrho(\sC | z)=0$.
\end{enumerate}
\end{lemma}
\begin{proof} We use the notation established in Subsection~\ref{ss:basic}.
(ii) $\Rightarrow$ (i): Suppose that $\varrho(\sC |z) =  0$. For $n\in \NN$, let
\[
U_n = \{x\mid\exists \text{ interval } I\colon x\in I \text{ and } \leb_I(\+C)<2^{-n}\}.
\]
Since $\+C$ is effectively closed, these classes are uniformly effectively open. By part~(1) of Lemma~\ref{lem:bounds}, $\leb(\+C\cap U_n)\leq 2^{-n+1}$, so $((U_{n+1})\sN n,\+C)$ is a difference test (see Definition~\ref{def:diff_test}). But clearly $z\in\+C\cap\bigcap_n U_n$, so $z$ fails this test.

(i) $\Rightarrow$ (ii): Suppose that $((U_n)\sN n,\+C)$ is a difference test that captures $z$. Because no ML-test captures $z$, it makes sense that this difference test must use its additional strength in the close vicinity of $z$. In other words, unless $\+C$ removes large parts of (some) $U_n$ near $z$, we could use the difference test to build a ML-test covering $z$. This is the idea we exploit.

We may assume that $U_n = \Cyl{D_n}$, where $(D_n)$ is a uniformly c.e.\ sequence of prefix-free sets of strings. By making this assumption, we are potentially removing dyadic rationals from the difference test. Since $z$ is not rational,   the difference test will still capture it.

Fix $r\in\NN$. We define a uniformly c.e.\ sequence of open classes $(G_m) \sN m $ with $\leb G_m\leq (1- \tp{-r-1})^m$ such that $z \not \in \bigcap_m G_m$ implies that there is a $\sss\prec z$ for which $\leb_\sss(\+C)\leq \tp{-r}$. Since this holds for every $r$, the hypothesis that $z$ is ML-random will imply that $\varrho(\sC | z ) =0$.

Let $G_0=[0,1]$. Suppose that $G_m$ has been defined and that $G_m=\Cyl{B_m}$ for some c.e.\ prefix free set $B_m$. We define $G_{m+1}$ as follows. Once $\sigma$ enters $B_m$, we declare
\[
G_{m+1} \cap \Cyl{\sss} = \left(U_{\sssl + r +1} \cap \Cyl{\sss} \right)^{(\le \tp{-\sssl}(1- \tp{- r- 1}))},
\]
where $\+W^{(\le \eps)}$ for a $\SI 1$ class $\+W$ is produced by running an enumeration for  $\+W$ but removing all strings from the enumeration that would make the total measure of enumerated strings 
exceed  $\eps$.
It is not hard to see that \bc $\leb G_{m+1}\leq (1- \tp{-r-1})\leb G_m\leq (1- \tp{-r-1})^{m+1}$ \ec and that $G_{m+1}=\Cyl{B_{m+1}}$ for some c.e.\ prefix-free set $B_{m+1}$.

Since $z$ is ML-random, there is a minimal $m > 0$ such that $z \not \in G_m$. The minimality of $m$ implies that there is a $\sss \in B_{m-1}$ with $\sss \prec z$. We show that $\leb_\sss(\+C)\leq \tp{-r}$.  Let $U  = U_{\sssl + r +1}$. Note that $\leb_\sss(U)> 1- \tp{-r-1}$, otherwise $z$ would enter $G_m$. Also $\leb_\sss(\+C\cap U) \le  2^{\sssl}\leb(\+C\cap U)\leq \tp{-r-1}$ by the definition of a difference test. But
\[
\leb_\sss(\+C) + \leb_\sss(U) - \leb_\sss(\+C\cap U)\leq 1,
\]
which implies that $\leb_\sss(\+C) \le \tp{-r}$, as required.
\end{proof}

\begin{remark} \label{rem:dyadic intervals} A basic dyadic interval has the form $[r \tp{-n}, (r+1) \tp{-n}]$ where $r \in \mathbb Z, n \in \mathbb N$.  The lower dyadic density of a set $\sC \subseteq \R$ at a point~$z$ is the variant  by only considering basic dyadic intervals containing $z$:  
$$\varrho_2(\sC | z):=\liminf_{z \in I \lland |I| \rightarrow 0} \frac{\lambda(I\cap \sC)}{|I|}, $$
where $I$ ranges over basic dyadic intervals. Clearly $\varrho_2(\sC | z) \ge \varrho(\sC | z)$. 

We note that the proof of  (i)$\to$(ii) actually shows that $\sC $ has lower dyadic density $0$ at $z$. Thus,   a  ML-random real which has positive dyadic density in every effectively closed class containing it already is a positive density point.
\end{remark}

\subsection{Density-one points}
\label{subsec:density}

We next address the question of how random $z$ must be to ensure that it is a density-one point. In other words, we want to ensure that if $\+C$ is an effectively closed class containing $z$, then $\varrho(\sC | z)=1$. By \cite{Day.Miller:nd},   difference randomness of $z$ is not enough.

First, note that it is sufficient to take $z$ to be weakly $2$-random. To see this, let $\eps>0$ be rational. Then $\varrho(\sC | z)<1-\eps$ if and only if
\[
\forall\beta>0\ \exists \gamma, \delta< \beta\colon \frac{\lambda([z-\gamma ,z+\delta] \cap \+C)}{\gamma + \delta} < 1-\eps,
\]
which is a $\Pi^0_2$ condition (because the measure of $\+C$ is approximable from above). By the Lebesgue density theorem, for each $\+C$ and $\eps>0$, the set of $z \in \+ C$ for which $\varrho(\sC | z)<1-\eps$ is null, so they are captured by a generalized ML-test.

The next lemma lets us give a somewhat more sophisticated partial answer. Informally, it says that a Martin-L\"of random that is not a density-one point must compute a fast growing function.

\begin{lemma}\label{lem:fast}
Let $z$ be  a Martin-L\"of random  real that is not a density-one point. Then $z$ computes a function $f\colon\N\to\N$ (that witnesses the fact that it is not a density-one point) such that: 
\begin{quote}
	for every oracle $A$, if $z$ is Martin-L{\"o}f random relative to $A$, \\
	then $f$ dominates every $A$-computable function.
\end{quote}
\end{lemma}

Taking $A=\emptyset$, the lemma shows that a Martin-L\"of random that is not a density-one point computes a function that dominates every computable function. Recall that Martin proved that $z$ is \emph{high} (i.e., $z'\geq_T\emptyset''$) iff $z$ computes a function that dominates every computable function.  Thus, 
if $z$ is a Martin-L\"of random and not a density-one point, then $z$ is high.

We can do significantly better. If $z$ is Martin-L\"of random, then by van Lambalgen's theorem, it is Martin-L\"of random relative to almost every $A$. Fix such an $A$. If $z$ is also not a density-one point, then the $f$ from the lemma dominates every $A$-computable function. In other words, $z$ is \emph{uniformly almost everywhere dominating}: it computes a function $f$ that dominates every $A$-computable function for almost every $A$. This property was introduced by Dobrinen and Simpson \cite{Dobrinen.Simpson:04}.
We may conclude:
\begin{theorem}\label{thm:aed}
If a Martin-L\"of random real $z$  is not a density-one point, then $z$ is uniformly almost everywhere dominating.
\end{theorem}

We call $z$ \emph{LR-hard}\footnote{The name LR-hard comes from \emph{LR-reducibility}. We write $A\leq_{LR}B$ \cite{Nies:AM} to mean that every set Martin-L\"of random relative to $B$ is Martin-L\"of random relative to $A$. In this notation, $\emptyset'\leq_{LR} z$ means exactly that $z$ is LR-hard.} if every set that is Martin-L\"of random relative to $z$ is $2$-random (i.e., Martin-L\"of random relative to $\emptyset'$). Kjos-Hanssen, Miller and Solomon \cite{KMS} proved that $z$ is (uniformly) almost everywhere dominating if and only if $z$ is LR-hard. Intuitively, such a real is ``nearly'' Turing above $\emptyset'$. For example, $z$ is \emph{superhigh} ($z'\geq_{tt}\emptyset''$)  by Simpson \cite{Simpson:07}.

Lemma~\ref{lem:fast} can also be used to reprove the  recent result of Bienvenu, Greenberg, \Kuc, Nies and Turetsky already discussed in Section~\ref{subsec:discussion} above. 
\begin{theorem}[\cite{Bienvenu.Greenberg.ea:OWpreprint}] \label{thm:ML-rd_nondensity}
If a Martin-L\"of random real $z$ is  not a density-one point, then~$z$ computes every $K$-trivial.
\end{theorem}

To see this, first assume that  $A$ is a c.e.\ $K$-trivial set. Then $A$ computes a function $g$ (its settling-time function) such that any function dominating $g$ computes~$A$. Since $A$ is $K$-trivial and therefore $\mathrm{Low}(\MLR)$, $z$ is   Martin-L\"of random relative to $A$ and, by   Lemma~\ref{lem:fast},   $z$ computes $A$.   For the general case we use the fact that every $K$-trivial set is computed by a c.e.\ $K$-trivial~\cite{Nies:AM}.

\begin{proof}[Proof of Lemma~\ref{lem:fast}]
Let $\eps<1$ and let $\+C\subseteq[0,1]$ be an effectively closed class containing $z$ such that $\varrho(\sC|z)<\eps$. We may assume, by adding suitable dyadic rationals to $\+C$, that there is a computable prefix-free sequence of strings $(\sigma_n)$ such that $\+C = [0,1]\smallsetminus\bigcup_n\Cyl{\sigma_n}$.

Let $\+C_{s,t} = [0,1]\smallsetminus\bigcup_{s\leq m<t} \Cyl{\sigma_m}$. Define $g(s)$ to be the least $t>s$ such that there is an interval $I$ containing $z$ for which $\leb(\+C_{s,t}|I)<\eps$. Note that $g$ is total, $z$-computable, and non-decreasing. Define $f\leq_T z$ by $f(n)=g^{\circ n}(0)$. In other words, let $f(0)=g(0)$ and, for all $n\in\N$, let $f(n+1)=g(f(n))$.

We will show that $f$ satisfies the lemma. To see this, assume that there is an $A$-computable function $h$ that is not dominated by $f$. We will use $h$ to build a Solovay test relative to $A$ that captures $z$. There are two cases.

\emph{Case $1$.} $h$ dominates $f$. We may assume, in this case, that $(\forall n)\; h(n)\geq f(n)$. Note that $(\forall n)\; f(n)\geq g(n)$. This is a property that we could have built into $f$ explicitly, but it actually follows from the definition. It is true for $n=0$. If it holds for $n$, then $f(n)\geq g(n)\geq n+1$, so $f(n+1)=g(f(n))\geq g(n+1)$. Therefore, $(\forall n)\; h(n)\geq g(n)$.

Define $k\leq_T A$ by $k(n)=h^{\circ n}(0)$ and, for all $n$, let
\[
S_n=\{x\mid\exists \text{ interval } I\colon x\in I \text{ and } \leb(\+C_{k(n),k(n+1)}|I)<\eps\},
\]
Note that $(S_n)_{n\in\N}$ is a uniformly $A$-c.e.\ sequence of open classes. By part (2) of Lemma~\ref{lem:bounds},
\[
\leb(S_n)\leq 2\frac{1-\leb\left(\+C_{k(n),k(n+1)}\right)}{1-\eps} = \delta\leb\Bigg(\bigcup_{k(n)\leq m<k(n+1)} \Cyl{\sigma_m}\Bigg),
\]
where $\displaystyle\delta = \frac{2}{1-\eps}$. Therefore,
\[
\sum_{n\in\N}\leb(S_n)\leq \delta\sum_{n\in\N}\leb\Bigg(\bigcup_{k(n)\leq m<k(n+1)} \Cyl{\sigma_m}\Bigg) = \delta\leb\Bigg(\bigcup_{m\in\N} \Cyl{\sigma_m}\Bigg)\leq \delta<\infty,
\]
where the equality follows from the fact that $(\sigma_n)$ is prefix-free. This shows that $(S_n)_{n\in\N}$ is an $A$-computable Solovay test.

We claim that  this test captures $z$. For every $n$, we have $k(n+1)=h(k(n))\geq g(k(n))$, so $\+C_{k(n),k(n+1)}\subseteq\+C_{k(n),g(k(n))}$. Thus there is an interval $I$ containing $z$ for which $\leb(\+C_{k(n),k(n+1)}|I)\leq \leb(\+C_{k(n),g(k(n))}|I)<\eps$, so $z\in S_n$.

\medskip
\emph{Case $2$.} $h$ does not dominate $f$. For all $n$, let
\[
S_n=\{x\mid\exists \text{ interval } I\colon x\in I \text{ and } \leb(\+C_{h(n),h(n+1)}|I)<\eps\}.
\]
As in the previous case, the sequence $(S_n)_{n\in\N}$ forms a Solovay test computable from $A$. We must show that it captures $z$. By our assumption, there are infinitely many $n$ such that $h(n)\leq f(n)$ and $h(n+1)\geq f(n+1)$. Fix such an $n$ and note that $h(n+1)\geq f(n+1) = g(f(n))\geq g(h(n))$. Therefore, there is an interval $I$ containing $z$ for which $\leb(\+C_{h(n),h(n+1)}|I)\leq \leb(\+C_{h(n),g(h(n))}|I)<\eps$, so $z\in S_n$. This is true for infinitely many $n$, so $z$ is not Martin-L\"of random relative to $A$.
\end{proof}

\subsection{Non-porosity points}

We now turn to the proof that difference randoms are non-porosity points. As we said above, this fact is an important part of the proofs of Theorems~\ref{thm:2} and~\ref{thm:1}. By the result of Day and Miller \cite{Day.Miller:nd}, there is a difference random that is not a density-one point. Thus, non-porosity is the strongest condition we can hope for. 

\begin{lemma}\label{lem:porous}
Let $\+C \sub [0,1]$ be an effectively closed class. If $z \in \+C$ is difference random, then $\+C$ is not porous at $z$.
\end{lemma}
\begin{proof}    
%

Suppose there is  $c \in \NN$ such that    $\+C$  is porous at $z$ via $\tp{-c+2}$. We will show that $z$ is not difference random.

For each string $\sss$ consider the set of minimal  ``porous'' extensions at stage $t$,
\begin{equation}\label{eqn_def_Nt}  N_t(\sss) = \left\{ \rho \succeq \sss \,\middle|  \,
\ex \tau \succeq \sss \left[\begin{array}{c}
   |\tau| = |\rho| \lland  |0.\tau - 0. \rho | \le \tp { - |\tau|  + c} \lland  \\  \Cyl \tau \cap \+C_t = \ES  \lland
 \rho \text{ is minimal with this property}
\end{array}\right]
\right\}. \end{equation}
%
%
In this proof, we say that a string $\sss$ meets $\+C$ if $\Cyl \sss \cap \+C \neq \ES$. We claim that
\begin{equation} \label{eqn:Ntbound} ~ ~ ~ ~ ~ ~ ~ ~ ~ ~  ~ ~ ~ ~ ~ ~  ~ ~ ~ ~ ~ ~ \sum_{\begin{subarray}{c}\rho \in N_t(\sss) \\ \rho \text{ meets } \+C\end{subarray}} \tp{-|\rho|} \le (1- \tp{-c-2}) \tp{-\sssl}. \end{equation}
To see this, let $R$ be the set of strings $\rho$ in (\ref{eqn:Ntbound}). Let $V$ be the set of    prefix-minimal strings   that occur as witnesses $\tau$ in~(\ref{eqn:Ntbound}). Then the open sets generated by $R$ and by $V$ are disjoint. Let $r$ and $v$ denote their measures, respectively. Since $R,V \sub \Cyl \sss$, we have $r+v \le \tp{-\sssl}$.  By definition of $N_t(\sss)$, for each $\rho \in R$ there is $\tau \in V$ such that  $|0.\tau - 0. \rho | \le \tp { - |\tau|  + c}$. This implies $r \le \tp{c+1}v$. The two inequalities together imply~(\ref{eqn:Ntbound})  because $r \le \tp{c+1}(\tp{-\sssl}-r)$ implies $r \le [1-1/(\tp{c+1}+1)]\tp{-\sssl}$.

Note that, by definition, even the \qt{holes} $\tau$ can be contained in the sets $N_t(\sss)$. This will be essential for the proof of the first of the following two claims. At each stage $t$ of the construction we define recursively a sequence of anti-chains  as follows.
\[
B_{0,t} = \{\estring\},\textnormal{ and for } n>0\colon B_{n,t} = \bigcup \{N_t(\sss)  \mid  \sss \in B_{n-1,t}\}.
\]

\noindent {\em Claim.}   If a string $\rho$  is in $B_{n,t}$ then it has a prefix $\rho'$ in $B_{n, t+1}$.

\noindent 	This is clear for $n=0$. Suppose  inductively that  it holds for $n-1$.  Suppose further that $\rho $ is in $B_{n,t}$ via a string $\sss \in B_{n-1,t}$. By the inductive hypothesis, there is a $\sss' \in B_{n-1, t+1}$ such that  $\sss' \preceq \sss$. Since $\rho \in N_{t}(\sss)$, $\rho$  is a viable extension of $\sss'$ at stage $t+1$ in the definition of $N_{t+1}(\sss')$,  except maybe for the  minimality. Thus there is $\rho' \preceq \rho $ in $N_{t+1}(\sss')$. \hfill $\Diamond$

\vsp

\noindent {\em Claim.}  For each $n,t$, we have
$ \sum \{ \tp{- |\rho| } \mid \rho \in  B _{n,t} \lland  \rho \ \text{\rm meets} \  \+C\} \le ( 1- \tp{-c-2})^n. $

\noindent This is again clear for $n=0$.  Suppose  inductively it holds for $n-1$.  Then, by (\ref{eqn:Ntbound}),


$$
\displaystyle\sum_{\begin{subarray}{c} \rho \in  B _{n,t} \\ \rho \text{ meets } \+C\end{subarray}} \tp{- |\rho| }   =  \displaystyle\sum_{\begin{subarray}{c}  \sss \in B_{n-1,t} \\ \sss \text{ meets } \+C\end{subarray}} \displaystyle\sum_{\begin{subarray}{c} \rho \in N_t(\sss) \\ \rho \text{ meets } \+C\end{subarray}} \tp{-|\rho|}
 \le   \displaystyle\sum_{\begin{subarray}{c} \sss \in B_{n-1,t} \\ \sss \text{ meets } \+C\end{subarray}} \tp{-\sssl } ( 1- \tp{-c-2})    \le   ( 1- \tp{-c-2})^n.$$
This establishes the claim.\hfill $\Diamond$

\vsp

Now let $U_n = \bigcup_t \Cyl {B_{n,t}}$. Clearly the sequence $(U_n) \sN n$ is uniformly effectively open. By the first claim, $U_n  = \bigcup_t  \Cyl {B_{n,t}}$ is a union over a nested sequence of classes, so the second claim implies that $\leb ( \+C \cap U_n) \le ( 1- \tp{-c-2})^n$.

\vsp

\noindent {\em Claim.} For all $n$, $z \in U_n$.

\noindent We show this by induction on $n$. We need to argue that for every initial segment $\sigma$ of $z$ we can find an extension $\sigma \preceq \rho \prec z$ such that $\rho$ is as in (\ref{eqn_def_Nt}) and there is a witness $\tau$ for $\rho$ as in (\ref{eqn_def_Nt}). Since $\+C$  is porous at $z$ via $\tp{-c+2}$,  along the binary expansion of~$z$ we will find infinitely many holes of the needed relative size that are good candidates for $\tau$. If we can argue that these holes can also be chosen as extensions of $\sigma$ then we have shown  that $\tau$ exists.

Clearly $z \in U_0$. If $n>0$, suppose inductively that there is a $\sss \prec z$ such that $\sss \in \bigcup_t B_{n-1,t}$. Since $z$ is random there are $\sigma^\prime,\eta$ such that $\sss \preceq \sigma^\prime1^{(c+1)}0^{(c+1)} = \eta \prec z$. We will choose
an extension of $\eta$ as our $\rho$, and will argue that then a $\tau$ witnessing this $\rho$ must be an extension of $ \sigma^\prime$ and therefore of $\sigma$. Because of our construction, $(0.\rho,0.\rho+2^{-|\rho|}) \subseteq (0.\eta,0.\eta+2^{-|\eta|}) \subseteq (0. \sigma^\prime,0. \sigma^\prime+2^{-| \sigma^\prime|})$.
There are two cases:

{\em Case $1$.} $0.\tau < 0.\rho$. By construction we have $0.\eta - 0. \sigma^\prime = 2^{-| \sigma^\prime|} \cdot (1 - 2^{-(c+1)})$. This implies
\[\begin{array}{rcl}
0.\tau - 0. \sigma^\prime & \geq & 0.\rho - 2^{-|\rho| + c} - 0.\sigma^\prime \\
              & \geq & 0.\eta  - 0.\sigma^\prime - 2^{-|\rho| + c}\\
              & \geq & 2^{-|\sigma^\prime|} \cdot (1 - 2^{-(c+1)}) - 2^{-|\eta| + c}\\
              & = & 2^{-|\sigma^\prime|} \cdot (1 - 2^{-(c+1)}) - 2^{-(|\sigma^\prime|+2c+2) + c}\\
              & \geq & 0.\\
\end{array}\]

{\em Case $2$.} $0.\rho < 0.\tau$. By construction we have $(0.\sigma^\prime + 2^{-|\sigma^\prime|}) - (0.\eta + 2^{-|\eta|}) = 2^{-|\sigma^\prime|} \cdot (2^{-(c+1)}-2^{-(2c+2)})$. This implies
\[\begin{array}{rcl}
(0.\sigma^\prime + 2^{-|\sigma^\prime|}) -   (0.\tau + 2^{-|\tau|}) & \geq & (0.\sigma^\prime + 2^{-|\sigma^\prime|}) - (0.\rho + 2^{-|\rho|+c}+ 2^{-|\rho|})\\
   & \geq & (0.\sigma^\prime + 2^{-|\sigma^\prime|}) - (0.\eta + 2^{-|\eta|}) -  2^{-|\rho|+c}\\
   & \geq & 2^{-|\sigma^\prime|} \cdot (2^{-(c+1)}-2^{-(2c+2)}) - 2^{-|\eta| + c}\\
   & = & 2^{-|\sigma^\prime|} \cdot (2^{-(c+1)}-2^{-(2c+2)}) - 2^{-(|\sigma^\prime|+2c+2) + c}\\
              & \geq & 0.\\
\end{array}\]

\noindent The two cases together imply $(0.\tau,0.\tau+2^{-|\tau|}) \subseteq (0. \sigma^\prime,0. \sigma^\prime+2^{-| \sigma^\prime|})$, as needed. This completes the claim.
%
%
Now take  a computable subsequence $(U_{g(n)}) \sN n$ such that $\leb ( \+C \cap U_{g(n)}) \le \tp{-n}$ to obtain a difference test that captures $z$.
\end{proof}

Theorem~\ref{thm:density_Turing} together with Lemma \ref{lem:porous} gives the following corollary. To the best of our knowledge, there exists no direct proof of this fact.

\begin{corollary}
Let $z$ be ML-random. If every effectively closed class $\+C$ with $z\in\+C$ satisfies  $\varrho(\sC|z) > 0$, then no effectively closed class $\+D$ with $z\in \+D$ is porous at $z$.
\end{corollary}

\section{Preliminaries: computable analysis}
\label{s:prelims_analysis}

We recall the definitions of computable and Markov computable functions on the real numbers. For more detail see for instance~\cite{Brattka.Hertling.ea:08,Weihrauch2000}.

\subsection*{Computable real-valued functions}

\begin{definition} \label{def:compreal}
A sequence of rational numbers $(q_n)_{n\in\mathbb{N}}$ is called a \emph{Cauchy name} for $z\in\R$ if   $|z-q_n|\leq 2^{-n}$ for all $n$.  We sometimes write $(z)_n$ for $q_n$ if a Cauchy name for $z$ is understood. 

A real number $z$ is \emph{computable} if it has a computable Cauchy name. \end{definition} 

Computability of a real  is equivalent to computability of the binary expansion. However,  one cannot \emph{uniformly} compute a binary expansion of a real from a Cauchy name. 
We denote the set of computable reals by~$\R_c$.

A partial function $f\colon\subseteq\R\to\R$ is \emph{computable} if there is a computable functional that takes every Cauchy name for $z\in\dom(f)$ to a Cauchy name for $f(z)$. In other words, there is an effective way to approximate $f(z)$ given approximations to $z$. It is not hard to see that a computable function is necessarily continuous on its domain. The use of Cauchy names instead of binary expansions is important. For example, there is no computable functional that takes a binary expansion of $z$ to a binary expansion of $3z$ for every $z\in\R$, but we clearly want $z\mapsto 3z$ to be a computable function. An equivalent and elegant way to define computability of partial functions from $\R$ to $\R$ is as follows: a function $f:D \rightarrow \R$ (with $D \subseteq \R$) is computable if and only if for every effectively open set of reals $\mathcal{U}$, one can compute, uniformly in an index for $\mathcal{U}$, an index for an effectively open set $\mathcal{V}$ such that $f^{-1}(\mathcal{U})=D \cap \mathcal{V}$. Note that for  computability of a function~$f$
one only needs to ensure the above for~$\mathcal{U}$ ranging over rational open intervals.

\subsection*{Markov computable real-valued functions}

Another approach to defining computability for functions on the real numbers is to consider only their action on computable reals. Let $(\varphi_e)_{e\in\N}$ be an effective numbering of the partial computable functions. An \emph{index name} for a computable real $z\in\R_c$ is an index for a computable Cauchy name for $z$. A partial function $f\colon \subseteq\R_c\to\R_c$ is \emph{Markov computable} if there is a uniform way to compute an index  name for $f(z)$ from an index name for $z\in\dom(f)$. More precisely, $f$ is Markov computable if there is a partial computable function $\nu\colon\N\to\N$ such that for all $z\in\dom(f)$, if $e$ is an index name for~$z$, then $\nu(e)$ is defined and is an index  name for~$f(z)$. Unless otherwise specified,  we will assume Markov computable functions are defined  on all of  $\R_c$; for definiteness, we will often say that such an  $f$ is total on $\R_c$.

Computable functions map computable reals to computable reals. So,  if $f$ is a computable function with domain containing~$\R_c$, then $f $ is a Markov computable function and total on $\R_c$. The converse is a special case of a theorem of Ce{\u\i}tin~\cite{Ce:59}: if $f\colon\R_c\to\R_c$ is a total Markov computable function, then there is a (possibly partial) computable function $\hat{f}\colon\subseteq\R\to\R$ such that $f=\hat{f}\uhr {\R_c}$. This implies that total Markov computable functions are continuous on $\R_c$, which is not obvious. Note that it is not always  possible to make $\hat{f}$ total. For example, in Theorem~\ref{thm:Markov} below we give an example of a Markov computable uniformly continuous function that is not obtained as the restriction to $\R_c$ of a total computable function. So the effective analysis of the Denjoy-Young-Saks theorem will be quite different for total computable functions and total Markov computable functions. It is also not true that every partial Markov computable function can be extended to a partial computable function.


Given a (Markov) computable function~$f$ and a real $z\in\dom(f)$, we sometimes use the notation $f(z)_n$ to denote the $n$th element in a Cauchy name for $f(z)$. Note that $|f(z)_n-f(z)|\leq 2^{-n}$.

\subsection*{An  extension theorem for nondecreasing functions}

Recall that $\sC \sub [0,1]$ is an \emph{effectively closed} class if  the complement of $\sC$ is the  union of  an effective   sequence of open intervals  $(a_i, b_i) \sN i$ with rational endpoints. We let $\sC_t = [0,1] \setminus \bigcup_{i \le t} (a_i, b_i)$; this is the approximation of $\sC$ at stage $t$. 

The Tietze extension theorem states that if $X$ is a sufficiently nice topological space (for example, a metric space), $\+C\subseteq X$ is closed, and $h\colon \+C\to\R$ is continuous, then $h$ can be extended to a continuous function on $X$. Weihrauch~\cite{MR1893086} proved an effective version of the Tietze extension theorem for computable metric spaces. The main result of this section is an effectivization of a variant of the Tietze extension theorem. In Theorem~\ref{thm:extension}, we prove that if $h\colon\+C\to\R$ is a nondecreasing computable function on an effectively closed class $\+C\subseteq[0,1]$, then it can be extended to a nondecreasing computable function on $[0,1]$.


%
%

\begin{lemma}\label{approx_lemma}
Let $p\colon [0,1]\to\R$ be a total computable function. Then the following functions are computable: \bc   $a, b \mapsto \sup_{a \le x \le b}p(x)$ and $a,b \mapsto \inf_{a \le x \le b}p(x)$,   \ec where  $0 \le a \le b \le 1$ . 
\end{lemma}
\begin{proof}   A real  number is  called lower semicomputable (resp.\ upper semicomputable) if it is the supremum (resp.\ infimum) of a computable sequence of rational numbers.  We prove the statement for the supremum; the proof for the infimum is analogous.  The reals $a, b$ are given by Cauchy names.    Let $S=\sup_{a \le x \le b}p(x)$. For a rational~$u$, we have
\[
S < u \Leftrightarrow [a,b] \cap p^{-1}([u,\infty)) = \emptyset.
\]
Note that the set  $[a,b] \cap p^{-1}([u,\infty))$ is effectively closed uniformly in the names for $a,b$. By compactness, if this set is   empty, this will become apparent at some stage. Thus, $S$ is an upper semicomputable real relative to the Cauchy names for $a,b$. 

For any rational $u$, we also have
\[
S > u \Leftrightarrow p^{-1}((u,\infty)) \cap [a,b] \not= \emptyset.
\] 
Since $p$ is continuous, there is a $\Sigma^0_1$ predicate, involving $u$ and Cauchy names for $a,b$, that holds if and only if $p^{-1}((u,\infty)) \cap [a,b] \neq \ES$.  Thus $S$ is also a lower semicomputable real relative to Cauchy names for $a,b$.  This establishes the computability of the function in question.
\end{proof}

A function $f\colon\subseteq\R\to\R$ is  called \emph{lower }[\emph{upper}]\emph{ semicomputable}   if there is a computable functional which, when given a Cauchy name for a real~$x$ as input,  enumerates a set of rationals whose supremum [infimum]  is $f(x)$.

\begin{lemma}\label{approx_lemma_unif}
Let $\+C$ be an effectively closed class. Let $h\colon \subseteq[0,1]\to\R$ be a computable function with domain containing $\+C$. Define   functions~$f$ and $g$   by
\[
f(a,b) = \inf  \{h(x)\colon \, x \in \+C \cap [a,b]\} ~ ~ ~ \text{ and } ~ ~ ~ g(a,b) = \sup\{h(x)\colon \, x \in \+C \cap [a,b]\},
\]
where $0 \le a \le b \le 1$. Then $f$  is lower semicomputable, and $g$ is upper semicomputable. 
\end{lemma}

\begin{proof}
Let $a,b \in [0,1]$. Let $(a_t)$ and $(b_t)$ be Cauchy names for $a$ and $b$. Since $\sC$ is closed and $h$ continuous, we  have for each rational $q$
\begin{eqnarray*}
g(a,b) < q & \leftrightarrow & \mathcal{C} \cap [a,b] \cap h^{-1}([q,\infty)) = \emptyset \\
 & \leftrightarrow  & \exists t\  \mathcal{C}_t \cap [a_t-2^{-t},b_t+2^{-t}] \cap h^{-1}([q,\infty))_t = \emptyset,
\end{eqnarray*}
where $h^{-1}([q,\infty))_t $ is the approximation at stage~$t$ of the effectively closed class $h^{-1}([q,\infty))$. Since one can decide effectively whether a boolean combination of rational closed intervals is empty, this shows that one can enumerate effectively, relative to  any pair of Cauchy names for $a$ and $b$,  the set of rationals  $q$ such that $g(a,b) < q$. 

The proof that $f$ is lower semicomputable is analogous. 
\end{proof}

While the next lemma appears to be folklore, a published proof is hard to find.
\begin{lemma}\label{lem:lower-smooth}
Suppose that a function $f\colon[0,1] \rightarrow \R$ is lower [upper] semicomputable. Then there exists a total computable function $F\colon [0,1] \times [0,\infty) \rightarrow \R$ such that $F$ is nondecreasing [nonincreasing] in its second argument, and for all~$x$, $\lim_t F(x,t) = f(x)$. 
\end{lemma}

\begin{proof}
Let us prove this for lower semicomputable $f$; the case of upper semicomputabile $f$  is analogous. By hypothesis, the set $U_r=f^{-1}((r,\infty))$ is effectively open.  Let  $(r_i)\sN i $ be  a fixed computable listing of the rationals. Let  $\chi_A$ denote  the characteristic function of a set  $A \sub [0,1]$. At    stage~$s$, we approximate $f$ by the function
\[
f_s = \max_{i \leq s}  \, (  r_i \cdot \chi_{U_{r_i}[s]}).
\]
 Each $U_{r_i}[s]$ is the  union of an effectively given  finite  collection $\mathcal A_{i,s}$ of rational open intervals $(a,b)$. Thus $\chi_{U_{r_i}[s]}$ cannot be computable, since it is not even continuous. The idea is then  to further approximate at stage~$s$ each characteristic function $\chi_{(a,b)}$ of the rational open interval $(a,b)$ by the function $g_{a,b,s}$ such that  $g_{a,b,s}(x) =0 $ for $ x \not \in [a,b]$, and    for  $x\in [a,b]$  
\[
g_{a,b,s}(x) =  \frac{s \cdot \min(|x-a|,|x-b|)}{1+s \cdot \min(|x-a|,|x-b|)}.
\] 
 This function is computable uniformly in~$s$. For each real $x$  its value is  nondecreasing in~$s$ and  tends to $\chi_{(a,b)}$. Thus  we can set
\[
F(x,s) = \max  \{r_i \cdot g_{a,b,s}(x) \mid  i \le s \lland (a,b) \in \mathcal A_{i,s}\}
\]
 with integer-valued second parameter. The second parameter can then be made real-valued by linear interpolation. 
\end{proof}

\begin{theorem}\label{thm:extension}
Let $h\colon\sub [0,1]\to\R$ be a computable function that is defined and non-decreasing on an effectively closed class $\+C$. Then $h\uhr {\+C}$ can be extended to a non-decreasing computable function $\hat{h}\colon [0,1]\to\R$.
\end{theorem}

\begin{proof}
 For all~$x$, let 
\[
f(x) = \sup_{z \in [0,x] \cap \+C} h(z) \ \ \ \ \text{and} \ \ \ \ g(x) = \inf_{z \in [x,1] \cap \+C} h(z).
\]
Since $h$ is non-decreasing on $\+C$, $f(x) \leq g(x)$ for all~$x$ and both $f$ and~$g$ are non-decreasing. Furthermore, by Lemma~\ref{approx_lemma_unif}, $f$ is upper semicomputable and $g$ is lower semicomputable. Thus by Lemma~\ref{lem:lower-smooth} we can write~$f$ and~$g$ as the limit of computable  functions $F$ and $G$ of two variables where $f(x)=\lim_t F(x,t)$ and $g(x)=\lim_t G(x,t)$ for all~$x$; the function  $F$~is nonincreasing in $t$, and~$G$ is nondecreasing in $t$. We may  further assume that $F$ and $G$ are both nondecreasing in $x$: otherwise,   using Lemma~\ref{approx_lemma} in relativized form, we may  replace $F(x,t)$ and $G(x,t)$ respectively by the computable functions $\widehat F(x,t) = \max_{0 \leq y \leq x} F(y,t)$ and $\widehat G(x,t)= \min_{x \leq y \leq 1} G(y,t)$.  The limit over~$t$ remains  unchanged by this operation: clearly,   $\lim_t \widehat F(x,t) \ge \lim_t F(x,t) = f(x)$. For the converse inequality, let  $q> f(x)$. Then, since $F$ is continuous and $f$ is nondecreasing,  the sets $\{ y \in [0, x]\colon \, F(y,n) < q \}$  ($n \in \NN$) form an open covering  of  $[0,x]$. Thus, there is $n$ such that $\widehat F(x, t) \le q$ for each $t \ge n$. 

%

We may also assume that $G(x,0)< F(x,0)$ for all~$x$,   because  $h$ is bounded and   hence  the values  at time $t=0$ can be chosen sufficiently large for $F$, and sufficiently small for $G$. Now define a  total function on $[0,1]$ by 
\[
\hat{h}(x) = F(x,t_x) = G(x,t_x),
\]
where $t_x$ is the smallest $t \in [0, \infty]$ for which $F(x,t) = G(x,t)$. Such a $t_x$ always exists as $f(x)\leq g(x)$ and $F$ and $G$ are continuous;  we need to allow  $t_x=\infty$ for the case that  $f(x)=g(x)$. 

We claim that $\hat{h}$ is as required. That is,  $\hat  h$ is computable, extends $h \uh \+C$,  and is non-decreasing. Firstly~$\hat{h}$ is computable since, in order to compute a rational within $2 \varepsilon$ of  $h(x)$, it suffices to find any~$t$ such that $G(x,t) <F(x,t) < G(x,t) +\varepsilon$, which  can be done effectively. Once $t$ is found we return a rational which is within $\varepsilon$ of  $F(x,t)$. 

We easily see that $\hat{h}$ extends~$h$.  When $x \in \+C$, the definition of $f$ and $g$, together with the fact that~$h$ is non-decreasing on~$\+C$, imply that $h(x)=f(x)=g(x)$.  In this case,  $t_x=\infty$, so  $\hat{h}(x)=F(x,\infty)=G(x,\infty)$, and  thus $\hat{h}(x)=f(x)=g(x)$. 

Finally, let us verify  that $\hat{h}$ is non-decreasing. Let $x<y$. There are two cases: 

\n (a) $t_x \leq t_y$. In this case
\[
\hat{h}(x)=G(x,t_x) \leq G(x,t_y) \leq G(y,t_y)=\hat{h}(y),
\]
using respectively the fact that $G$ is nondecreasing in its second argument and nonincreasing in its first one.

\n  (b)  $t_y \leq t_x$. In this case
\[
\hat{h}(x)=F(x,t_x) \leq F(x,t_y) \leq F(y,t_y)=\hat{h}(y),
\]
using respectively the fact that $F$ is nonincreasing in its second argument and nondecreasing in its first one. 
\end{proof}

\section{Effective forms of the Denjoy-Young-Saks theorem}
\label{sec:denjoy}

\subsection*{Derivatives, pseudo-derivatives, and the Denjoy alternative}

We start by defining the various (pseudo)-derivatives that we need for our work on the Denjoy-Young-Saks Theorem. For a function~$f\colon\subseteq\R\to\R$, the \emph{slope} at a pair $a,b$ of distinct reals in its domain is
\[
S_f(a,b) = \frac{f(a)-f(b)}{a-b}.
\]
If $z$ is in an open neighborhood of the domain of~$f$, the \emph{upper} and \emph{lower derivatives} \label{def_upper_lower_deriv} of $f$ at $z$ are
\[
\ol D f(z)  =  \limsup_{h\ria 0} S_f(z, z+h) \quad  \textnormal{and}    \quad
\underline D f(z)  =  \liminf_{h\ria 0} S_f(z, z+h),
\]
where as usual, $h$ ranges over positive and negative values. The derivative $f'(z)$ exists if and only if these values are equal and finite.

If $f$ is a Markov computable function, then $\ol D f(z)$ and $\underline D f(z)$ are not defined because the domain of $f$ only contains computable reals. Nonetheless, if $\dom(f)$ is dense, one can consider the upper and lower \emph{pseudo}-derivatives \label{def_pseudo_deriv} defined by:  
\begin{align*}
\utilde Df(x) &= \liminf_{h \to 0^+} \, \{S_f(a,b)  \mid a, b \in \dom(f)   \lland  \, a\le x \le b \lland\, 0 <  b-a\le h\}, \\
\widetilde Df(x) &= \limsup_{h \to 0^+} \,  \{S_f(a,b)  \mid a, b \in \dom(f)  \lland  \, a\le x \le b \lland\, 0 <  b-a\le h\}.
\end{align*}
If $f$ is continuous on its (dense) domain, which is the case for computable and for total Markov computable functions, then one can replace $\dom(f)$ by any dense subset of $\dom(f)$ in the definitions of  $\utilde  Df$ and $\widetilde  Df$. For Markov computable functions, for example, one could use~$\Q$ instead of $\R_c$ to define the pseudo-derivatives. It is well known (see e.g.\  \cite[Fact 7.2]{Brattka.Miller.ea:nd})  that for continuous functions with domain $[0,1]$, the lower and upper pseudo-derivatives of $f\uh\Q$ coincide with the usual lower and upper derivatives. 

We are ready for the formal definition of the Denjoy alternative.

\begin{definition} \label{def:DA}
Let~$f\colon\subseteq [0,1] \rightarrow \R$ be a partial function with dense domain, and let $z \in [0,1]$. We say that~$f$ satisfies the \emph{Denjoy alternative} at $z$ if~
\begin{itemize}
\item either the pseudo-derivative of $f$ at $z$ exists (meaning that $\widetilde D f(z) =\utilde D f(z)$),
\item or $\widetilde D f(z) = +\infty$ and $\utilde D f(z) = - \infty$.
\end{itemize}
\end{definition}

Intuitively this means that   either the function behaves well near~$z$ by having a derivative at this point, or it behaves badly in the worst possible way:  the limit superior and the limit inferior are as different as possible. The Denjoy-Young-Saks theorem (see, e.g., Bruckner~\cite{MR507448}) states that the Denjoy alternative holds at almost all points for {\em any} function~$f$.

\subsection{Characterizing computable randomness via the Denjoy alternative for computable functions}
\label{ss:DenjoyComputable_random}
Recall that a Markov computable function~$g$ has domain containing $\R_c$ unless otherwise mentioned.
\begin{definition}[Demuth \cite{Dem88preprint}]\label{df:DenjoyRandom}
A real $z \in [0,1]$ is called \emph{Denjoy random} (or a \emph{Denjoy set}) if $\utilde D g(z) < +\infty$ for every Markov computable function $g$.
\end{definition}

In a preprint by Demuth \cite[p.\ 6]{Dem88preprint}  it is shown that
if  $z \in [0,1]$ is  Denjoy random, then for every \emph{computable} $f\colon [0,1] \to \mathbb R$, the Denjoy alternative holds at~$z$. This material was rediscovered and made accessible by \Kuc{}. See \cite[Def.\ 2]{Kucera.Nies:11} for more background and references on the relevant work of Demuth.

In combination with the results in \cite{Brattka.Miller.ea:nd}, we have the following
pleasing characterization of computable randomness through a  differentiability property of computable functions.

\begin{theorem}\label{thm:DenjoyCR}
The following are equivalent for a real $z \in [0,1]$.\begin{enumerate}
\item $z$ is Denjoy random.
\item $z$ is computably random.
\item For every computable $f \colon [0,1] \to \mathbb R$, the Denjoy alternative holds at~$z$.
\end{enumerate}
\end{theorem}
\begin{proof}
(1) $\Rightarrow$ (3) is a result of Demuth \cite{Dem88preprint}.

 \vsps

(3) $\Rightarrow$ (2): Let $f$ be a non-decreasing computable function. Then $f$ satisfies the Denjoy alternative at $z$. Since $\utilde Df(z) \ge 0$, this means that $f'(z)$ exists.
This implies that  $z$ is computably random by Brattka et al.~\cite[Thm.\  4.1]{Brattka.Miller.ea:nd}.

 \vsps

(2) $\Rightarrow$ (1):
Assume that $z$ is not Denjoy random. In other words, there is a Markov computable function $g$ such that $\utilde Dg(z) = + \infty$. Given a binary string $\tau$, we write $S_g(\tau)$ to mean $S_g(a,b)$ where $(a,b) = \Cyl \tau$. Note that $S_g(\tau)$ is a computable real uniformly in $\tau$. Furthermore, the function $\tau \mapsto S_g(\tau)$ satisfies the martingale equality, and succeeds on $z$ in the sense that its values are unbounded (and even converge to $\infty$) along~$z$. However,  this function may have negative values; informally, this is as if we were allowed to ``bet with debt'' because we can increase our capital  at a string $\tau 0$ beyond $2S_g(\tau)$ by incurring a debt, i.e., negative value, at $S_g(\tau 1)$. We can use $S_g$ to build a proper computable martingale that succeeds on $z$.

If $z$ is computably random, it is not a dyadic rational. So there is a string $\sss\preceq z$ such that if $\Cyl \sss = (a,b)$, then $S_g(r,s) > 2$ for all $r,s$ such that $z \in (r,s) \sub (a,b)$.

\vsps

\emph{Case $1$.} There are infinitely many $\tau$ with $\sigma\preceq\tau\prec z$ such that there is a $v\in\{0,1\}$ for which $S_g(\tau v)_0\leq 1$. ($S_g(\tau v)_0$ is the first term in the Cauchy name for the computable real $S_g(\tau v)$, which is at most $1$ away from that real.) For such a $\tau$, we have $S_g(\tau v)\leq 2$. By our choice of $\sigma$, we know that $z$ is not an extension of $\tau v$.

We define a computable martingale $M$ on extensions $\tau\succeq\sss$ that will succeed on (the binary expansion of)~$z$. Let $M(\sigma) = 1$. Suppose now that  $M(\tau)$ has been defined and is positive. If $\tau$ and $v$ are as above, let $M(\tau u) = 2M(\tau)$, where $u=1-v$, and let $M(\rho) = 0$ for all $\rho\succeq\tau v$. If not, let $M(\tau v) = M(\tau)$ for $v\in\{0,1\}$. By assumption, along~$z$, the martingale $M$ doubles its value infinitely often  and never loses value. So it succeeds on~$z$.

\vsps

\emph{Case $2$.} Otherwise, assume without loss of generality that no such $\tau$ and $v$ exists. Again, we build a computable martingale $M$ on extensions $\tau\succeq\sss$. Let $M(\sigma) = S_g(\sigma)$. If $M(\tau)$ has been defined and is positive, check if there is a $v\in\{0,1\}$ for which $S_g(\tau v)_0\leq 1$. If so, it is safe to let $M(\rho) = M(\tau)$ for all $\rho\succeq\tau$ because we know that $\tau\nprec z$. If not, let $M(\tau v) = S_g(\tau v)$ for $v\in\{0,1\}$. Thus $M$ is a computable martingale and $M$ and $S_g$ agree along $z$, so $M$ succeeds on $z$.

\vsps

In both cases, $z$ is not computably random.
\end{proof}

Note that all we need for (2) $\Rightarrow$ (1) is that $f(q)$ is a computable real uniformly in a rational $q \in [0,1] \cap \Q$. Thus, in Definition~\ref{df:DenjoyRandom}  we can replace the Markov computability of $f$ by this   hypothesis, which seemingly leads to a stronger randomness notion.

\subsection{Weak 2-randomness yields the Denjoy alternative for functions computable on the rationals} 
First we review some definitions and facts from the last section of \cite{Brattka.Miller.ea:nd}.    
Let $I_\QQ = [0,1] \cap \QQ$. A function $f \colon \sub [0,1] \to \RR$ is called \emph{computable on $I_\QQ$} if $f(q)$ is defined for each $q \in I_\QQ$, and $f(q)$ is  a computable real  (see Definition~\ref{def:compreal}) uniformly in~$q$.

 For any rational $p$,  let 
	\bc $\utilde C(p) = \{z \colon \,  \fa t >0 \, \ex a,b  [ a \le z \le b \lland 0< b-a \le t \lland   \, S_f(a,b ) < p \}$,  \ec where $t,a,b$ range over rationals.
Since $f$ is computable on $I_\QQ$, the set  
\bc $\{z \colon \,  \ex a,b  \,  [ a \le z \le b \lland 0< b-a \le t \lland   \, S_f(a,b ) < p\}$ \ec
  is a $\SI 1$ set  uniformly in $t$.	Then $\utilde C(p)$ is $\PI 2$ uniformly in $p$. Furthermore, 
\begin{equation} \label{eqn:CDtilde}\utilde Df(z) <p \RA z \in \utilde C(p) \RA \utilde Df(z) \le p.  \end{equation}
	  Analogously we define 
	\bc $\widetilde C(q) = \{z \colon \,  \fa t >0 \, \ex a,b  [ a \le z \le b \lland 0< b-a \le t \lland   \, S_f(a,b ) >  q \}$. \ec
	
Similar observations hold for these sets.

\begin{theorem} \label{thm:DA_w2r}  Let $f \colon \sub [0,1] \to \RR$ be computable on $I_\QQ$. Then $f$ satisfies the Denjoy alternative at every weakly $2$-random real $z$. \end{theorem} 

\begin{proof}  We adapt   the {classical proof}   in \cite[p. 371]{Bogachev.vol1:07} to the case of pseudo-derivatives.   We analyze the arithmetical complexity of exception sets in order to conclude that weak $2$-randomness is sufficient for the Denjoy alternative to hold. 

   We let $a,b, p , q$ range over  $I_\QQ$.  Recall Definition~\ref{def:compreal}.  For each 
 $r< s$, $r,s \in I_\QQ$, and for each $n \in \NN$, let  
  \begin{equation} \label{eqn:Enrs} E_{n,r,s} = \{ x \in [r,s] \colon \, \fa a,b [ r \le a \le x \le b \le s \to S_f(a,b)_0 >  -n+1]\}. \end{equation} 
 Note that $E_{n,r,s}$  is a  $\PI 1$ class. For every $n$ we have the implications 
 
 \bc $\utilde Df(z) > -n+2 \to \ex r,s \,  [ z \in E_{n,r,s}]  \to \utilde Df(z) > -n. $ \ec
 
To show the Denjoy alternative of Definition~\ref{def:DA} at   $z$, we may assume that  $\utilde D f(z) > - \infty$ or $\widetilde Df(z) < \infty$. If the second condition holds we replace $f$ by $-f$, so we may assume   the first condition holds. Then   $z \in E_{n,r,s}$ for some    $r,s, n$ as above.  Write $E= E_{n,r,s} $. 


For $p< q$, the class $E \cap  \utilde C(p) \cap  \widetilde C(q)$ is $\PI 2$. By (\ref{eqn:CDtilde}) it suffices to show that each such  class is null. For this, we show that  for a.e.\ $x \in E$, we have $\utilde Df(x) = \widetilde Df(x)$. 
 This remaining part of  the argument is entirely  within classical analysis. Replacing  $f$ by $f(x) + nx$, we may assume that for $x \in E$, we have  \bc $\fa a,b [ r \le a \le x \le b \le s \to S_f(a,b)_0 >  1]$.  \ec
 
 Let $f_*(x) = \sup_{r\leq a \le x} f(a)$. Then $f_*$ is  nondecreasing on $(r,s)$. Let $g$ be an arbitrary  nondecreasing function defined on $[0,1]$ that extends $f_*$. Then by a  classic theorem of Lebesgue, $L(x)  := g'(x)$ exists for a.e.\ $x \in [0,1]$. 
 
%
Recall porosity from Definition~\ref{def:porous at}.
 By the Lebesgue density theorem, the points in $E$  at which $E$ is porous form a null set. 
 
 \begin{claim} \label{cl:nonporous_does_it} For each $x \in E$ such that $L(x)$ is defined and $E$ is not porous at $x$, we have $\widetilde Df(x) \le L(x) \le \utilde Df(x) $. \end{claim} 
 
  Since  $\utilde Df(x) \le \widetilde Df(x)$, this establishes the theorem.
  
 To prove the claim, we show $ \widetilde Df(x) \le L(x)$, the other inequality being symmetric.  Fix $\epsilon >0$.  Choose $\alpha > 0$ such that 
 \begin{equation} \label{eqn:uxv} \fa u,v \in E \, [ (u \le x \le v \lland 0< v - u \le \alpha )  \to S_{f_*}(u,v) \le L(x) (1+\epsilon)]; \end{equation}
 furthermore, since $E$ is not porous at $x$, for each $\beta \le \alpha$, the interval $(x-\beta, x+ \beta)$ contains no open subinterval of length $\epsilon \beta$ that is disjoint from $E$. 
 Now suppose that $a,b \in I_\QQ$, $a < x< b$ and   $\beta = 2(b-a) \le \alpha$. There 
 are $u,v \in E$ such that $0\le a-u \le \epsilon \beta$ and $0 \le v-b \le \epsilon \beta$. Since $u,v \in E$ we have $f_*(u) \le f(a) $ and $f(b) \le f_*(v)$. (Only the former relies on the definition of $E$; the latter is immediate from the definition of $f_*$.) Therefore $v-u \le  b-a + 2 \epsilon \beta  = (b-a) (1 + 4 \epsilon)$. It follows that 
  \[ S_f(a,b) \le \frac {f_*(v)  - f_*(u)}{b-a} \le S_{f_*}(u,v)(1+4\epsilon) \le L(x) (1+4\epsilon)(1+\epsilon). \qedhere\]
  \end{proof}

\subsection{Markov computable functions satisfy the Denjoy alternative at all difference random reals}
\label{subsec:diff-Denjoy}

We now turn to the proof of Theorem~\ref{thm:2}. We derive it from a result of interest on its own, which is formulated  in terms  of non-porosity. Recall again  that Markov computable functions are defined on all of $\R_c$. 
\begin{theorem}  \label{prop:cr-nonpor-da}
Let~$z$ be a computably random real that is also a non-porosity point. Then $z$ is DA-random, i.e., all Markov computable functions satisfy the Denjoy alternative at~$z$.
\end{theorem}

Theorem~\ref{thm:2} now follows because every difference random real is computably random, and, by Lemma~\ref{lem:porous}, a non-porosity point.

\begin{proof}   Note that each Markov computable function is computable on $I_\QQ$. We will adapt the proof of the foregoing Theorem~\ref{thm:DA_w2r}  to the stronger hypothesis  that  the given function $f$ is   Markov computable, in order to show that the weaker present hypothesis on the real $z$ is now sufficient for the Denjoy alternative.

Given $n\in \NN$ and $r<s$ in $I_\QQ$, define the set $E = E_{n,r,s}$ as above. As before, we may assume that for $x \in E$, we have $\fa a,b \, [ r \le a \le x \le b \le s \to S_f(a,b)_0 >  1]$, and hence the  function  $f_*(x) = \sup_{a \le x} f(a)$ is nondecreasing on $(r,s)$, hence on $E$.

We will invoke Theorem~\ref{thm:extension} in order to show that some total nondecreasing extension $g$ of $h= f_*\uhr E$ can be chosen to be computable.   
%
\begin{claim} \label{porous_diff_restate}  The function $f_*\uhr E$ is computable. \end{claim}  
To see this, recall that $p,q$ range over $I_\QQ$, and  let $f^*(x) = \inf_{x\leq q\le s} f(q)$. If $x\in E$ and $f_*(x) < f^*(x)$ then $x$ is computable: fix a rational $d$ in between these two values. Then $p< x \lra f(p) < d$, and $q> x \lra f(q) > d$. Hence $x$ is both lower and upper semicomputable, and therefore computable.
Now a Markov computable function is continuous at every computable $x$. Thus $f_*(x)= f^*(x)$ for each $x$ in $E$.

To compute $f_*(x)$ for $x \in E$ up to precision $\tp{-n}$, we can now simply search for rationals $p<x <q$ such that $0< f(q)_{n+2} - f(p)_{n+2} < \tp{-n-1}$, and output $f(p)_{n+2}$. If during this search we detect that  $x \not \in E$, we stop. This shows the claim.

Now, by Theorem~\ref{thm:extension}, a  total nondecreasing extension $g$ of $h= f_*\uhr E$ can be chosen to be computable.   
By \cite[Thm.\  4.1]{Brattka.Miller.ea:nd} we know that $L(x) : = g'(x)$ exists for each computably random  real $x$.  
Since $z$ is a non-porosity point, invoking Claim~\ref{cl:nonporous_does_it}  concludes the proof of the theorem. 
\end{proof}

\def\sF{\ensuremath{\mathcal{F}}\xspace}
\begin{remark}\label{remark_banach_mazur} Let \sF be the class of continuous  functions $f\colon\subseteq [0,1] \rightarrow \R$ with  domain containing  $\R_c$ such that  $f(q)$ is a computable real uniformly in a rational $q$. It is easy to check that Claim~\ref{porous_diff_restate} already holds for functions in \sF: to verify the claim note that the $f_*(x)=\sup\{f(a)\mid r\leq a\leq x, a \in \Q\}$. Hence the proof that  $f_*$ 
is computable on the relevant interval works under the weaker hypothesis $f\in\sF$.

Recall that a function $f\colon\subseteq [0,1] \rightarrow \R$ with  domain containing  $\R_c$ is called Banach-Mazur computable if it maps every computable sequence of reals to a computable sequence of reals (but not necessarily uniformly). Mazur \cite{Mazur_63_Continu} proved that these functions are continuous (for a more general version in computable metric spaces see Hertling \cite[Theorem 16]{Hertling2000}). Thus all Banach-Mazur computable functions satisfy the Denjoy alternative at difference randoms.

Hertling \cite{hertling_personal} showed that   the Banach-Mazur computable functions form a proper subclass of  
\sF.
\end{remark}

\subsection{The class $\mathsf{DA}$  is incomparable with the Martin-L\"of random reals}

First we give  a real $x$ that is DA-random but not ML-random. By Theorem~\ref{prop:cr-nonpor-da}, it is enough to prove:

\begin{theorem}\label{forcing_thm_proper}
There exists a computably random real~$x$ that is a density-one point and not Martin-L\"of random.
\end{theorem}

\begin{corollary}
There exists a real~$x$ that is not Martin-L\"of random and all Markov computable functions satisfy the Denjoy alternative at $x$.
\end{corollary}

\begin{proof}[Proof of Theorem \ref{forcing_thm_proper}]
We present the construction in the language of forcing. The forcing partial order $\PP$ is inspired by the well-known construction due to Schnorr \cite{Sc:71} of a computably random real $x$ that is not ML-random. Schnorr's idea was to consider more and more computable martingales along the binary expansion of $x$, and ensure that an appropriate linear combination of the finitely many martingales currently considered cannot increase too much. One could say that $\PP$ is  a direct paraphrase  of Schnorr's argument in the language of forcing. Suprisingly, a sufficiently generic filter yields a real that is not ML-random and also a  density-one point. 

Let $\PP$ be the   set of conditions of the form $\uple{\sigma,M,q}$ where $M$ is a computable  martingale, $\sigma$ a string and $q$ a rational, and such that $M(\sigma)<q$.  We say that $\uple{M',\sigma',q'}$ \emph{extends} $\uple{\sigma,M,q}$, which we write $\uple{\sigma',M',q'} \leq \uple{\sigma,M,q}$, if
\begin{itemize}
\item $q' \leq q$,
\item $\sigma \preceq \sigma'$ and $M(\tau) < q$ for all~$\sigma \preceq \tau \preceq \sigma'$,
\item For all $\tau \extend \sigma'$, $M'(\tau) < q' \Rightarrow M(\tau)< q$.
\end{itemize}
To each condition $\uple{\sigma,M,q}$ we associate the effectively closed set (of positive measure)
\[
L_\uple{\sigma,M,q} = \{x \in \cs \mid x \extend \sigma \wedge (\forall\tau)~ \sigma \preceq \tau \prec x \rightarrow M(\tau) < q\}.
\]
(Notice that $\uple{\sigma',M',q'} \leq \uple{\sigma,M,q}$ implies $L_\uple{\sigma',M',q'} \subseteq L_\uple{\sigma,M,q}$.)\\

\noindent We claim that for a sufficiently generic filter $G \subseteq \PP$, the closed set
\[
\bigcap_{\uple{\sigma,M,q} \in G} L_\uple{\sigma,M,q}
\]
is a singleton~$\{x\}$ and~$x$ is not Martin-L\"of random but satisfies the Denjoy alternative for every Markov computable function. We verify this fact via a succession of claims.

\vsp

\noindent {\em Claim $1$.} For any filter $G \subseteq \PP$, the  set $\G = \bigcap_{\uple{\sigma,M,q} \in G} L_\uple{\sigma,M,q}$ is non-empty. If~$G$ is sufficiently generic, $\G$ is in fact a singleton~$x$ which is equal to the union of the strings appearing in the conditions of~$G$. \\
\noindent {\em Subproof.}
By compactness, if~$\G$ is empty, then there are finitely many conditions $\uple{\sigma_i,M_i,q_i}$ such that $\bigcap_{i} L_\uple{\sigma_i,M_i,q_i}$ is empty. Since $G$ is a filter, let $\uple{\sigma^*,M^*,q^*}$ be a condition in~$G$ extending all the $\uple{\sigma_i,M_i,q_i}$. We have $\bigcap_{i} L_\uple{\sigma_i,M_i,q_i} \supseteq L_ \uple{\sigma^*,M^*,q^*}$, and the latter is non-empty as it has positive measure.

Now, for all~$n$, let~$D_n$ to be the set of conditions $\uple{\sigma,M,q}$ with $|\sigma| \geq n$. One can see that $D_n$ is dense: indeed, for a condition $\uple{\sigma,M,q}$, if $|\sigma| < n$, then diagonalizing against~$M$ during~$n -|\sigma|$ steps, one can find an extension~$\tau$ of $\sigma$ of length at least~$n$ such that $\uple{\tau,M,q}$ extends~$\uple{\sigma,M,q}$. Therefore if~$G$ is sufficiently generic, it contains conditions $\uple{\sigma,M,q}$ for arbitrarily long~$\sigma$. Since~$G$ is a filter, all the strings appearing in its elements must be comparable, hence there is a unique real~$x$ that extends them all. Therefore~$\G$ contains at most the singleton~$\{x\}$. Since $\G$ is non-empty, it is equal to the singleton~$\{x\}$.
\hfill $\Diamond$

\vsp

From now on we assume that $G$ is generic enough to ensure that $\G$ is a singleton, which will be called~$\{x\}$.

\vsp

\noindent {\em Claim $2$.}
If~$G$ is sufficiently generic, and $\uple{\sigma,M,q}$ is a condition in~$G$, then ${M(x \uh n) < q}$ for all~$n \geq |\sigma|$.

\noindent {\em Subproof.}
Trivial by definition of~$x$. \hfill $\Diamond$

\vsp

\noindent {\em Claim $3$.}
If~$G$ is sufficiently generic, then $x$ is computably random.

\noindent {\em Subproof.}
This is the usual argument. Let~$N$ be a computable martingale which we can assume to have initial capital~$1$. Let $\uple{\sigma,M,q}$ be a condition. Let $\delta$ be a positive rational such that $M(\sigma)< q - \delta$. Then, setting $M'= M+2^{-|\sigma|}\delta N$, it is easy to see that $\uple{\sigma,M',q}$ extends~$\uple{\sigma,M,q}$, and that $N$ does not succeed on any element of $L_\uple{\sigma,M',q}$.
\hfill $\Diamond$

\vsp

\noindent {\em Claim $4$.}
If~$G$ is sufficiently generic, then $x$ is not Martin-L\"of random.

\noindent {\em Subproof.}
This part is also quite standard. Let $\uple{\sigma,M,q}$ be a condition and let~$c$ be a constant. Since one can computably diagonalize against a computable martingale, there exists a computable sequence~$z$ extending~$\sigma$ such that $M(z \uh n) < q$ for all~$n > |\sigma|$. Since $z$ is computable, it is possible to take $n$ large enough so that $\tau =z \uh n$ satisfies $K(\tau) < |\tau| - c$. This proves that a sufficiently generic~$G$ will yield a sequence~$x$ that is not Martin-L\"of random.
\hfill $\Diamond$

\vsp

\noindent {\em Claim $5$.}
If~$G$ is sufficiently generic, then $x$ is a point of density~$1$ of every $\Pi^0_1$ class $\+C$ of positive measure to which it belongs.

\noindent {\em Subproof.}
Fix $\eps >0$. We want to extend any given condition $\uple{\sigma,M,q}$ to a new condition that forces either $x \notin\+C$ or $\varrho(\sC|x)\geq 1-\eps$ (assuming that $x$ is computably random, which we showed in Claim 3 to be ensured by sufficient genericity). We may assume without loss of generality that as soon as the martingale $M$ reaches a capital greater than~$q+1$, it stops betting; formally: if  $M(\tau)_0> q+1$,  then $M(\rho) = M(\tau) $ for each $\rho \succeq \tau$. (Here $M(\tau)_0$ denotes the first member of the Cauchy name that $M$ computes from $\tau$. Note that $M(\tau)_0> q+1$ implies $M(\tau)> q$.) Indeed, we can transform $M$ into a computable martingale $M'$ with this additional property, which clearly ensures $\uple{\sigma,M',q} \leq \uple{\sigma,M,q}$.


 The advantage of this assumption is that whenever we find a string $\tau \extend \sigma$ such that $M(\tau) < q$, then we immediately know that $\uple{\tau,M,q}$ is an extension of $\uple{\sigma,M,q}$.

Now, if there is a $\tau\succeq\sigma$ such that $\Cyl{\tau}\cap\+C=\emptyset$ and $M(\tau)<q$, then $\uple{\tau,M,q}$ is a valid extension of  $\uple{\sigma,M,q}$ that forces $x\notin\+C$. Assume that no such $\tau$ exists, that is, for each $\tau \succeq \sigma$ with $\Cyl \tau \cap \+ C = \emptyset$ we have $M(\tau) \ge q$.  Note that $\Cyl{\sigma}\smallsetminus\+C$ can be expressed as $\bigcup_{i\in\omega} \Cyl{\tau_i}$ for an appropriate collection of strings $\{\tau_i\}_{i\in\omega}$ extending~$\sigma$ (in this section, we view cylinders $\Cyl{\tau}$ as \emph{closed} subintervals of $[0,1]$). For each~$i$, we have $M(\tau_i)\geq q$. So $\Cyl{\sigma}\smallsetminus\+C$ is a subclass of $\+D = \Cyl{\{\tau\succeq\sigma\mid M(\tau)\geq q\}}$. Since  $x$  is computably random, $x$ is not an endpoint of $\Cyl{\sigma}$. Hence   we have $\varrho(\sC|x)\geq \varrho(\Cyl{\sigma}\smallsetminus\+D|x)$. Our goal is to force the latter to be at least $1-\eps$.

Let $d= \inf_{\sigma'\succeq\sigma} M(\sigma')$. Note that $d < q$ and choose $r,s \in\Q$ and a $\tau\succeq\sigma$ such that
\begin{itemize}
	\item $d \leq M(\tau)<r<s$, and
	\item $s-d\leq \eps (q-d)$.
\end{itemize}
This can be done by taking $\tau \extend \sigma$ to be such that  $M(\tau)$ is very close to $d$. Informally, what we are doing here is identifying the ``savings" of $M$ at $\sigma$, which is precisely $\inf_{\sigma'\succeq\sigma} M(\sigma')$: this is money that $M$ cannot lose, but it cannot use it for further betting either. Then we pick an extension $\tau$ such that $M(\tau)$ is close to $d$, meaning that $M(\tau)$ is---apart from its savings---almost broke at $\tau$. 

Consider the condition $\uple{\tau,M,r}$, which is an extension of $\uple{\sigma,M,q}$. We claim that if $G$ contains this condition and $x$ is computably random, then $\varrho([0,1]\smallsetminus\+D|x)\geq 1-\eps$. Proving this would be easier if we were only concerned with the dyadic intervals containing $x$, and we would not need $x$ to be computably random for that case. To handle arbitrary intervals, we take a detour through the result of Brattka, Miller and Nies \cite[Thm.\ 4.1]{Brattka.Miller.ea:nd} mentioned in the introduction: since $x\in[0,1]$ is computably random, every computable non-decreasing function is differentiable at $x$.

For $\rho\in 2^{<\omega}$, recall that $0.\rho$ and $0.\rho+2^{-|\rho|}$ are the left and right endpoints of $\Cyl{\rho}\subseteq[0,1]$, respectively. Define a real-valued  function $f$ on the dyadic rationals in $\Cyl{\tau} $ by  $f(0.\tau)=0$ and, for all $\rho\succeq\tau$,
\[
 S_f (0.\rho, 0.\rho + \tp{-|\rho|})= M(\rho).
\]
Since $M$ is bounded, $f$ is Lipschitz on its domain. Since $f(p)$ is a computable real uniformly in a dyadic rational $p$, it is clear that $f$ can be extended to a computable function on $\Cyl \tau$ (also denoted $f$). Note that $f$ is  non-decreasing. 


We are assuming that $x$ is computably random, so $f$ is differentiable at $x$. The fact that $\uple{\tau,M,r}\in G$ implies that if $\tau\preceq\rho\preceq x$, then $S_f (0.\rho, 0.\rho + \tp{-|\rho|}) = M(\rho) < r$, so $f'(x)\leq r<s$.

Now consider an open interval $(a,b)$ containing $x$. We know that $x$ is not an endpoint of $\Cyl{\tau}$, so if $b-a$ is sufficiently small, then $(a,b)\subseteq\Cyl{\tau}$ and $S_f(a,b)<s$. We claim that, for such an interval, $\lambda_{\+D}(a,b)\leq \eps$. Assume otherwise. Let $\{\tau_i\}_{i\leq m}$ be a finite prefix-free collection of strings such that $\Cyl{\tau_i}\subseteq (a,b)$ and $M(\tau_i)\geq q$, for all $i\leq m$, and such that for  $Q=\bigcup_{i\leq m} \Cyl{\tau_i}$, we have    $\lambda (Q)>(b-a)\eps$. This implies that
\[
(b-a)s > f(b)-f(a)\geq \lambda(Q)q + ((b-a)-\lambda(Q))d.
\]
Hence $(b-a)(s-d) > \lambda(Q)(q-d) > (b-a)\eps(q-d)$, which contradicts our choice of $s$ and $d$. Therefore, $\lambda_{\+D}(a,b)\leq \eps$ for any sufficiently small interval $(a,b)$ containing $x$. This means that $\rho(\+C |x )\geq \rho(\Cyl{\sigma}\smallsetminus\+D|x)\geq 1-\eps$, as required.
 \hfill $\Diamond$

\vsp

\noindent This completes the proof.
\end{proof}

Next, we show that some  \ML\ random real  is not DA-random. This result is due to Demuth \cite{Demuth:76}; see \cite[Cor.\ 10]{Kucera.Nies:11} for some background. Since Demuth's notation and proofs are very hard to access, it is worth providing a construction in modern language. It also shows that the counterexample can be lower semicomputable.
\begin{theorem} \label{thm:Markov}
There exists a Markov computable function~$f$ for which the Denjoy alternative does not hold at  a lower semicomputable ML-random real. Moreover,~$f$ can be taken to be uniformly continuous, i.e., it can be built in such a way that it has a (unique) continuous extension to $[0,1]$.
\end{theorem}
\begin{proof}

Recall that there is  a universal Martin-L\"of test $(U_n)\sN n$, namely,    the set of  reals in $[0,1]$ that are not in $\MLR$ coincides with $\bigcap_n U_n$. Since no computable real  is Martin-L\"of random, every $x \in \R_c$ belongs to $U_1$. 

Let $\alpha$ be the leftmost point of the complement of $U_1$. Since $U_1$ is an effectively open class, it is an effective union $\bigcup_t I_t$ of \emph{closed} rational intervals $I_t$ that intersect at most  at their endpoints. Let $U_1[s] = \bigcup_{t<s} I_t$ and let $\alpha_s$ be the leftmost point of $[0,1]\smallsetminus U_1[s]$. Then $\alpha$ is approximated from below by the computable sequence of rationals $(\alpha_s)\sN s$.

Our function~$f$ is defined as the restriction to $\R_c$ of the following function $F$. Outside $U_1$, $F$ is equal to $0$. On $U_1$, it is constructed sequentially as follows. At stage~$s+1$, consider $I_s$. There are two cases.
\begin{enumerate}
\item Either adding this interval does not change the value of $\alpha$ (i.e., $\alpha_{s+1}=\alpha_s$). In that case, define the function~$F$ to be equal to zero on $I_s$.
\item Or, this interval does change the value of $\alpha$: $\alpha_{s+1}>\alpha_s$. In this case, define $F$ on $I_s$ to be the triangular function taking value $0$ on the endpoints of $I_s$ and reaching the value $v$ at the middle point, where~$v$ is defined as follows. Let~$t$ be the last stage at which the \emph{previous} increase of $\alpha$ occurred (i.e., $t$ is maximal such that $t<s$ and $\alpha_{t+1} > \alpha_t$). Let $n$ be the smallest integer such that the real interval~$[\alpha_{t},\alpha_{t+1}]$ contains a multiple of $2^{-n}$. For that~$n$, set $v=2^{-n/2}$.
\end{enumerate}

First, we see that the restriction~$f$ of $F$ to $\R_c$ is Markov computable: given a potential  index name~$e$ for a computable real $x$ in the sense of Section~\ref{s:prelims_analysis}, we try to compute the real~$x$ coded by~$e$  until we find a sufficiently good estimate $a < x < b$ such that the interval $[a,b]$ is contained either in one or in the union of two of the intervals appearing in the enumeration of $U_1$. It is then easy to compute~$F$ at~$x$ as one can decide which of the above cases hold for each interval, and both the zero function and the triangular function are computable on $\R_c$. (In the ``triangular'' case, note that the value~$n$ of the construction can be found effectively.)

We claim that the function~$f$ does not satisfy the Denjoy alternative at~$\alpha$. More precisely, we have $\widetilde D f(\alpha) = 0$ and $\utilde D f (\alpha) = - \infty$. Notice that $f$ is equal to $0$ on $(\alpha,1] \cap \R_c$ and non-negative on $[0, \alpha) \cap \R_c$, taking the value~$0$ at computable reals arbitrarily close to $\alpha$ (at least the endpoints of intervals~$I_s$ enumerated on the left of $\alpha$), therefore $\widetilde D f(\alpha) = 0$ is clear. To see that $\utilde D f (\alpha) = - \infty$, given~$k\in \NN$ let $b_k$ be the dyadic real  which is a multiple of $2^{-k}$, is smaller than $\alpha$ and such that $\alpha - a_k < 2^{-k}$. Since $a_k < \alpha$, there exists a stage~$t$ such that $a_k \in [\alpha_t, \alpha_{t+1}]$. Let $s>t$ be the next stage at which $\alpha$ increases. By definition, $F$ is then defined to be a triangular function on $[\alpha_s,\alpha_{s+1}]$ of height $2^{-k/2}$.  
Thus, letting $x_k$ be middle point of $[\alpha_s,\alpha_{s+1}]$ and $q>\alpha$ be a rational such that $q-b_k<2^{-k}$, we have
$$
S_f(x_k,q) = \frac{f(q)-f(x_k)}{q-x_k} \leq \frac{0-2^{-k/2}}{2^{-k}} = - 2^{k/2}.
$$
Since this happens for all~$k$, we have $\utilde D f(\alpha)= -\infty$.

It remains to show that the function~$F$ is continuous on $[0,1]$. But this is almost immediate as one can write $F = \sum_n h_n$, where $h_n$ is the function equal to $0$ except on the intervals on which $F$ is a triangular function of height $2^{-n/2}$, and on that interval $h_n = F$. It is obvious that the $h_n$ are continuous and $||h_n|| \le 2^{-n/2}$. Therefore $\sum_n ||h_n|| < \infty$, so by the Weierstrass M-test we can
conclude that the convergence is uniform and hence the function $\sum_n h_n$ is continuous.
\end{proof}

\bibliographystyle{amsplain}
\providecommand{\bysame}{\leavevmode\hbox to3em{\hrulefill}\thinspace}
\providecommand{\MR}{\relax\ifhmode\unskip\space\fi MR }
\providecommand{\MRhref}[2]{%
  \href{http://www.ams.org/mathscinet-getitem?mr=#1}{#2}
}
\providecommand{\href}[2]{#2}

\end{document}